\documentclass[final,reqno,twoside,a4paper,11pt]{amsart}
\usepackage{mltools}
\usepackage{mlmath62,mlthm10-REAR}
\usepackage{upref,amsmath,amssymb,amsthm,mathrsfs}
\usepackage{tikz,graphics,latexsym}

\usepackage[colorlinks=true,linkcolor=blue,citecolor=blue]{hyperref}

\usepackage[scaled]{helvet} % ss
\usepackage{courier} % tt
\usepackage[mathbf]{euler}
\usepackage{mllocal} 

\usepackage{pgfplots}
\numberwithin{equation}{section}

\begin{document}

\date{\today}

\title[Brasselet number and Newton polygons]{Brasselet number and Newton polygons}

\author{Tha\'is M. Dalbelo}
\address{Department of Mathematics, 
	Universidade Federal de S\~ao Carlos (UFSCar),
	Brazil}
\email{thaisdalbelo@dm.ufscar.br, thaisdalbelo@gmail.com}

\author{Luiz Hartmann}
\address{Department of Mathematics, 
	Universidade Federal de S\~ao Carlos (UFSCar),
	Brazil}
\email{hartmann@dm.ufscar.br}
\urladdr{http://www.dm.ufscar.br/profs/hartmann}

\thanks{Luiz Hartmann is Partially support by FAPESP: 2016/16949-8 and both 
authors
are partially supported by CAPES/PVE:88881.068165/2014-01 }


\begin{abstract}
	We present a formula to compute the Brasselet number of $f:(Y,0)\to (\C, 
	0)$ where $Y\subset X$ is a non-degenerate complete intersection in a 
	toric variety $X$. As applications we establish 
	several results concerning invariance of the Brasselet number for 
	families of non-degenerate complete intersections. 
	Moreover, when $(X,0) = (\mathbb{C}^n,0)$ we derive sufficient conditions 
	to obtain the invariance of the Euler obstruction for families 	of 
	complete intersections with an isolated singularity which are contained in 
	$X$.  		 
\end{abstract}

\maketitle

\tableofcontents

%%%%%%%%%%%%%%%%%%%%%%%%%%%%%%%%%%%%%%%%%%%%%%%%%%%%%%%%%%%%%%%
\section{Introduction}
%%%%%%%%%%%%%%%%%%%%%%%%%%%%%%%%%%%%%%%%%%%%%%%%%%%%%%%%%%%%%%%
Given a germ of an analytic function 
$f:(\mathbb{C}^{n},0) \to (\mathbb{C},0)$ with an isolated critical point at 
the 
origin, an important invariant of this germ is its Milnor number \cite{M1}, 
denoted by $\mu(f)$. The Milnor number is considered as a central invariant, 
since 
it provides algebraic, topological and geometric information about the germ 
$f$. For instance, the Milnor number coincides with the number of Morse points 
of a morsefication of $f$.

\medskip

Initially the Milnor number was associated to germs of analytic functions $f 
: (\mathbb{C}^n, 0) \to (\mathbb{C}, 0)$ with an isolated critical point, and 
consequently it was used to study isolated hypersurface singularities. However
this invariant is well defined in many others contexts, for example curves
\cite{BuG}, isolated complete intersection singularities (ICIS) \cite{Hamm}, 
and determinantal varieties of codimension two \cite{MC}, to name just a few.

\medskip

The local Euler obstruction was defined by MacPherson in \cite{Mac} for 
the construction
of characteristic classes of singular complex algebraic varieties. 
Thereafter, it has been
deeply investigated by many authors such as Brasselet and Schwartz \cite{BS}, 
Dutertre \cite{N}, Gaffney, Grulha and Ruas \cite{GGR}, Gon\-za\-lez-Sprinberg 
\cite{Gonzalez}, L\^{e} and Teissier \cite{LT},  Matsui and Takeuchi \cite{MT2}, among others.
We denote by $(X,0)$ a germ of an analytic singular space embedded in 
$\mathbb{C}^n$ and by $f:(X,0) \to (\mathbb{C},0)$ a germ of an analytic 
function 
with an isolated critical point at the origin. Brasselet, 
Massey, Parameswaran and Seade \cite{BMPS} introduced an invariant associated 
to $f$ 
called the Euler obstruction of $f$  and denoted by ${\rm 
Eu}_{f,X}(0)$. Roughly speaking, ${\rm Eu}_{f,X}(0)$ is the obstruction to 
extending a lifting of the conjugate of the gradient vector field of $f$ as a 
section of the Nash bundle of $(X , 0)$. This invariant is closed related 
with the local Euler obstruction of $X$, what explains its name.  

\medskip

An important consequence of the definition given by MacPherson, is that the local Euler obstruction is a constructible function,    
which means that, it is constant along the strata of a Whitney stratification 
of $X$. This is essentially a consequence of the topological            
triviality of $X$ on the Whitney strata. As a consequence, the local Euler 
obstruction 
does not depend on the Whitney stratification of $X$. The following 
Lefschetz-type formula was proved by 
Brasselet, L\^e and
Seade \cite{BLS}.

% proved first in \cite{BS} and then by several other 
%authors

% Then, given $(X,0) \subset (\mathbb{C}^n,0)$ an equidimensional complex analytic singularity
%	germ with a Whitney stratification $\{V_{i}\}$, we will denote by ${\rm Eu}_{X}(V_i)$ the value of the Euler
%	obstruction of $X$ at any point of the stratum $V_i$.

\begin{theorem}\label{BLS}
	Let $(X,0) \subset (\mathbb{C}^n,0)$ be an equidimensional complex analytic singularity
	germ with a Whitney stratification $\{V_{i}\}$, then
	given a generic linear form $L$, there exists $\varepsilon_0$
	such that for any $\varepsilon$ with $0<\varepsilon<
	\varepsilon_0$, we have
	\begin{equation*}
	{\rm Eu}_X(0)=\sum_{i}\chi \big(V_i\cap B_\varepsilon\cap 
	L^{-1}(\delta) \big) \cdot
	{\rm Eu}_{X}(V_i),
	\end{equation*}
	where $\chi$ is the Euler-Poincar\'e
	characteristic, $B_\varepsilon:=B_\varepsilon(0)$ is the ball 
	with center the origin and radius $\varepsilon$, ${\rm Eu}_{X}(V_i)$ is the 
	value of the local Euler
	obstruction of $X$ at any point of the stratum $V_i$,
 and $0 < \vert 
	\delta \vert \ll \varepsilon \ll 1$.
\end{theorem}

The previous theorem says that the local Euler obstruction, as a 
constructible 
function
on $X$, satisfies the Euler condition relatively to a generic linear function.

\medskip

For the Euler obstruction of an analytic function $f:(X,0) \to 
(\mathbb{C},0)$ with an isolated critical point at the origin, there is also 
a 
Lefschetz-type formula. This formula was proved in \cite{BMPS}. The purpose 
of the authors was to understand what prevents the
local Euler obstruction from satisfying the local Euler condition with respect to
functions which are singular at the origin.
\begin{theorem}\label{BMPS}
	Let $(X,0) \subset (\mathbb{C}^n,0)$ be an equidimensional complex analytic singularity
	germ with a Whitney stratification $\{V_{i}\}$, and let $ f : (X,0)  \to 
	(\mathbb C,0)$ be a function with an isolated
	singularity at $0$. Then,
	$${\rm Eu}_{f,X}(0)={\rm Eu}_X(0) \,- \,\left(\sum_{i} \chi 
	\big(V_i\cap B_\varepsilon\cap
	f^{-1}(\delta) \big) \cdot  {\rm Eu}_X(V_i) \right),$$
	where $0 < \vert \delta \vert \ll \varepsilon \ll 1$.
\end{theorem}
\noindent The last equation presents the relation between the local 
Euler 
obstruction of $X$ and 
the Euler obstruction of $f$.

\medskip

Seade, Tib{\u a}r and Verjovsky continued the study of the properties of  
${\rm Eu}_{f,X}(0)$ in \cite{STV}. The authors proved that the Euler obstruction of $f$ is
closely related to the number of Morse points of a morsefication of $f$, as 
follows.

\begin{proposition}
	Let $(X,0)$ be an equidimensional complex analytic singularity
	germ of dimension $d$ and $f: (X,0) \to (\mathbb{C},0)$ a germ of an analytic function 
	with an isolated critical point at the origin. Then 
	$$
	{\rm Eu}_{f,X}(0) = (-1)^{d} n_{\rm reg},
	$$ 
	where $n_{\rm reg}$ is the number of Morse points 
	in the regular part of $X$ appearing in a stratified morsefication of $f$.
\end{proposition}
\noindent Therefore, the Euler obstruction of $f$ is the number of Morse points 
of a morsefication of $f$ on the regular part of $X$, up to the sign. Hence 
this 
 invariant can be seen as a generalization of the Milnor number of $f$. 
 
\medskip 
 
Another invariant associated with a germ of an analytic function \linebreak 
$f:(X,0) \to (\mathbb{C},0)$ is the Brasselet number introduced by Dutertre 
and Grulha in \cite{NN}. We will denote this number by ${\rm B}_{f,X}(0)$. If 
$f$ 
has an isolated critical point, then the Brasselet number satisfies the 
equality 
\begin{equation*}
{\rm B}_{f,X}(0) = {\rm Eu}_{X}(0) - {\rm Eu}_{f,X}(0).
\end{equation*}
\noindent If $f$ is linear and generic, then ${\rm B}_{f,X}(0)={\rm 
Eu}_X(0)$, hence it can be viewed as a 
generalization of the local Euler obstruction. Moreover, 
even if $f$ has a non-isolated singularity it provides interesting results. 
For example, the Brasselet number has a L\^e-Greuel type formula 
(see \cite[Theorem 4.4]{NN} or Theorem \ref{Le Greuel} below), \ie 
the difference of the Brasselet numbers ${\rm B}_{f,X}(0)$ and ${\rm 
B}_{f,X^{g}}(0)$ is measure by the number of Morse critical points on the top 
stratum of the Milnor fiber of $f$ appearing in a morsefication of $g$, where 
$g: (X,0) \to (\mathbb{C},0)$ is a prepolar function (see Definition 
\ref{prepolar}) and $X^{g}=X \cap g^{-1}(0)$.

\medskip

Although they are important, the invariants mentioned above are not easily 
computed using their definition. In the literature there are formulas which 
make the computation easier, see \cite{BLS, BMPS, LT, N}. Some authors worked 
on 
more specific situations. In the special case of toric surfaces,
an interesting formula for the local Euler obstruction was proved by 
Gonzalez-Sprinberg \cite{Gonzalez}. This formula was generalized by Matsui 
and Takeuchi \cite{MT2} for normal toric varieties of any dimension. 

\medskip

Toric varieties are particularly interesting objects, since they have a 
strong relation with elementary convex geometry. On these varieties we have 
an action of the algebraic torus $(\mathbb{C}^{*})^n$ that induces a finite 
decomposition of the variety into orbits,
all of which are homeomorphic to a torus.

\medskip

In \cite{Varchenko}, Varchenko described the topology of the Milnor fiber of a 
function $f:(\mathbb{C}^n,0) \to (\mathbb{C},0)$ using the geometry of the 
Newton polygon of $f$, and consequently, the Milnor number can be expressed 
by volumes of polytopes related to the Newton polygon of $f$. In his proof, 
he constructed a toric modification of $\mathbb{C}^n$ on which the pull-back 
of $f$ defines a hypersurface with only normal crossing singularities. While 
$\mathbb{C}^n$ is a very special smooth toric variety, it seems natural to 
generalize his formula to Milnor fibers over general singular toric 
varieties. This was done by Matsui and Takeuchi in \cite{MT1}.

\medskip

We use \cite{MT1} to establish several combinatorial formulas for the 
computation of the Brasselet number of $f:(Y,0)\to (\C, 0)$ where $Y\subset X$ 
is a non-degenerate complete intersection in a toric variety $X$. These 
formulas will be given in terms of volumes of Newton polygons associated to 
$f$.

\medskip

This paper is organized as follows. In Section $2$, we present some background 
material
concerning the Brasselet number and toric varieties, which will be used in the 
entire work. In Section $3$, we compute the Brasselet number of a polynomial 
function $f : (X, 0) \to (\mathbb{C}, 0)$, where $X \subset \mathbb{C}^n$ is a 
toric variety. Moreover, we compute this invariant for functions defined on 
$X^{g} = X\cap g^{-1}(0)$, where $g: X \to \mathbb{C}^k$ is a non-degenerate 
complete intersection. As a consequence, assuming that $g$ has an isolated 
critical point on $X$ and on $X^f$, we also obtain a formula for the number of 
stratified Morse critical points on the top stratum of the Milnor fiber of $f$ 
appearing in a morsefication of $g: X \cap f^{-1}(\delta) \cap 
B_{\varepsilon} \to \mathbb{C}$. As applications we establish several 
results concerning constancy of these invariants. In Section $4$, we 
consider the case where $(X,0) = (\mathbb{C}^n,0)$ and we derive sufficient 
conditions to obtain the constancy of the Euler obstruction for families of 
complete intersections with an isolated singularity which are contained on $X$. 
We 
use this result to study the invariance of the Bruce-Roberts' Milnor number 
for 
families of functions defined on hypersurfaces. In Section $5$, we work in the 
case of surfaces,\ie in the case where $X$ is a $2$-dimensional
toric variety. In this situation, we present a 
characterization of a polynomial function $g: X \to \mathbb{C}$ which 
has a stratified isolated singularity at the origin. We use  
this characterization to present some examples for a class of toric surface 
that is also determinantal. In Section \ref{Section-IndicesVectorFields}, we 
study the GSV-index on non-degenerated complete intersection in toric varieties.

%%%%%%%%%%%%%%%%%%%%%%%%%%%%%%%%%%%%%%%%%%%%%%%%%%%%
\section{Generalities: stratifications, Brasselet number and toric varieties}
%%%%%%%%%%%%%%%%%%%%%%%%%%%%%%%%%%%%%%%%%%%%%%%%%%%%

For the convenience of the reader and to fix the notation we present some 
general facts  in order
to establish our results.

%%%%%%%%%%%%%%%%%%%%%%%%%%%%%%%%%%%%%%%%%%%%%%%%%%%%
\subsection{Stratifications and Brasselet number}
%%%%%%%%%%%%%%%%%%%%%%%%%%%%%%%%%%%%%%%%%%%%%%%%%%%%

In order to introduce the definition and the properties of the Brasselet 
number, 
we need some notions about stratifications. For more 
details, we refer to Massey \cite{Massey,Massey1}. Let $A \subset \C^n$, 
$B\subset \C^m$ and $f 
: A \to 
B$ be a function with $n, m\in \N$. We will 
fix the notation, $A^{f}:=A\cap f^{-1}(0)$. 

Consider $X\subset \C^n$ a 
reduced complex analytic set of dimension $d$ which is included in an open 
set $U$. Let $F : U \to \mathbb{C}$ be a holomorphic function and $f : X \to 
\mathbb{C}$ be the restriction 
of $F$ to $X$, \ie $f:=F|_X$.

\begin{definition}\label{Def-GoodStratification}
	A good stratification of $X$ relative to $f$ is a stratification 
	$\mathcal{V}$ of $X$ which is adapted
	to $X^{f}$, such that $\left\{V_i \in 
	\mathcal{V}; \ \ V_i \not\subset X^{f}\right\}$ is a Whitney 
	stratification
	of $X \setminus X^{f}$, and for any pair of strata 
	$(V_{\alpha},V_{\beta})$ such that $V_{\alpha} \not\subset X^{f}$ and 
	$V_{\beta} \subset X^{f}$, the $(a_f)$-Thom condition is satisfied. We 
	call the strata included in $X^{f}$ the good strata.
\end{definition}

By \cite{Le2}, given a stratification $\mathcal{S}$ of $X$ one can refine 
$\mathcal{S}$ to obtain a Whitney stratification $\mathcal{V}$ of $X$ which 
is adapted to $X^{f}$. By \cite{Massey1} the refinement 
$\mathcal{V}$ satisfies the  $(a_f)$-Thom condition, \ie good 
stratifications always exist.

\begin{definition}
Consider $X$ and $f$ as before. Let $\mathcal{V}=\{V_i\}$ be a stratification 
of $X$. The critical locus of $f$ 
relative to $\mathcal{V}$, denoted by $\Sigma_{\mathcal{V}} f$, is the union 
of the 
critical locus of
$f$ restricted to each of the strata,\ie $\Sigma_{\mathcal{V}} f = \bigcup_{i} \Sigma(f|_{V_i})$.
\end{definition}

A {\it {critical point of $f$ relative to $\mathcal{V}$}} is a point $p \in 
\Sigma_{\mathcal{V}} f$. If the stratification $\mathcal{V}$ is clear, we
will refer to the elements of $\Sigma_{\mathcal{V}} f$ simply as stratified 
critical points of $f$. If $p$ is an isolated point of $\Sigma_{\mathcal{V}} 
f$, we call $p$ a {\it{ stratified isolated critical point of $f$}} (with
respect to $\mathcal{V}$). If $\mathcal{V}$ is a Whitney stratification of 
$X$ and $f : X \to \mathbb{C}$ has a stratified isolated
critical point at the origin, then 
\begin{equation*}
\left\{V_{\alpha} \setminus X^{f}, 
V_{\alpha} \cap X^{f} \setminus \left\{0 \right\}, \left\{ 0\right\}; \ \ 
V_{\alpha} \in \mathcal{V} \right\},
\end{equation*}
is a good stratification for $f$. We call it the {\it 
good 
stratification induced by $f$}.

\begin{definition}\label{BrasseletN}
	Suppose that $X$ is equidimensional. Let $\mathcal{V} = \left\{V_i 
	\right\}_{i=0}^{q}$ be a good stratification of $X$ relative
	to $f$. The Brasselet number is defined by 
	\begin{equation*}
	{\rm B}_{f,X}(0) := \sum_{i 
	=1}^q\chi \big(V_i\cap B_{\varepsilon}(0) \cap f^{-1}(\delta) \big) \cdot
	{\rm Eu}_{X}(V_i),
	\end{equation*}
	where $0< \left| \delta \right| \ll \varepsilon \ll 1$.
\end{definition}

If $f$ has a stratified isolated critical point at the origin and $X$ 
is equidimensional, Theorem \ref{BMPS} implies that
\begin{equation}\label{EqBrasEul}
{\rm B}_{f,X}(0) = {\rm Eu}_X(0) - {\rm Eu}_{f,X}(0).
\end{equation}
The Brasselet number has many interesting properties,\eg it satisfies
several
multiplicity formulas, which enable the authors to establish in \cite{NN} a relative version of the local
index formula and a Gauss-Bonnet formula for ${\rm B}_{f,X}(0)$. However, one 
of the most important properties 
of this invariant is the L\^e-Greuel type formula (Theorem \ref{Le 
Greuel}). To present this result we need 
to impose some conditions on the functions to ensure that $X^{g}$ meets 
$X^{f}$ in a nice way. So it is necessary to define:

\begin{definition}\label{prepolar}
	Let $\mathcal{V}$ be a good stratification of $X$ relative to $f$. We say 
	that $g : (X,0) \to (\mathbb{C}, 0)$
	is prepolar with respect to $\mathcal{V}$ at the origin if the origin is 
	an isolated critical point of $g$.
\end{definition}

The condition that $g$ is prepolar means that $g$ has an isolated critical 
point (in the stratified sense), both on $X$ and on $X^{f}$, and that $X^g$ 
intersects transversely each stratum of $\mathcal{V}$ in a neighborhood of 
the origin, except
perhaps at the origin itself. However, it is important to note that, while 
$X^g$ meets $X^f$ in a nice way, $X^f$ may have
arbitrarily bad singularities when restricted to $X^g$. The $(a_f)$-Thom 
condition in Definition 
\ref{BrasseletN} together with the hypothesis that $g$ is prepolar ensure 
that $g:X \cap f^{-1}(\delta) \cap B_{\varepsilon} \to \mathbb{C}$ has no 
critical points on $g^{-1}(0)$ \cite[Proposition $1.12$]{Massey}. 
Therefore, the number of stratified Morse critical points on the top 
stratum $V_q \cap f^{-1}(\delta)\cap B_{\varepsilon}(0)$, in a	morsefication 
of $g : X \cap f^{-1}(\delta) \cap B_{\varepsilon}(0) \to 
\mathbb{C}$, does not depend on the morsefication.

The next theorem shows that the Brasselet number satisfies a L\^e-Greuel 
type formula \cite[Theorem 4.4]{NN}. 

\begin{theorem}\label{Le Greuel}
	Suppose that $X$ is equidimensional and that $g$ is prepolar with respect 
	to $\mathcal{V}$ at the
	origin. Then,
	\[
	{\rm B}_{f,X}(0) - {\rm B}_{f,X^{g}}(0)  = (-1)^{d-1}n_q,
	\]
	where $n_q$ is the number of stratified Morse critical points on the top 
	stratum $V_q \cap f^{-1}(\delta)\cap B_{\varepsilon}(0)$ appearing in a 
	morsefication of $g : X \cap f^{-1}(\delta) \cap B_{\varepsilon}(0) \to 
	\mathbb{C}$, and $0 < \left|\delta\right| \ll \varepsilon \ll 1$. In particular, this number is independent on the morsefication.
\end{theorem}

%%%%%%%%%%%%%%%%%%%%%%%%%%%%%%%%%%%%%%%%%%%%%%%%%%%%%%%%%%%%%%%%%%%%%%%%%%%%%%%%%%%%%%%%%
%%%%%%%%%%%%%%%%%%%%%%%%%%%%%%%%%%%%%%%%%%%%%%%%%%%%%%%%%%%%%%%%%%%%%%%%%%%%%%%%%%%%%%%%
\subsection{Toric varieties}\label{Subsectio-Toricvarieties}
%%%%%%%%%%%%%%%%%%%%%%%%%%%%%%%%%%%%%%%%%%%%%%%%%%%%%%%%%%%%%%%%%%%%%%%%%%%%%%%%%%%%%%%%%

The theory of toric varieties can be seen as a
cornerstone for the interaction between combinatorics and algebraic geometry, 
which relates the combinatorial study of convex polytopes to algebraic 
torus actions. Moreover, for polynomial functions defined on such varieties, 
it is possible to obtain a combinatorial description of the topology of their 
Milnor fibers in terms of Newton polygons (see \cite{MT1,Oka}). The reader 
may consult \cite{F,O} for an overview about toric 
varieties.

Let $N \cong \mathbb{Z}^d$ be a $\mathbb{Z}$-lattice of rank $d$ and $\sigma$ 
a strongly convex rational polyhedral cone in $N_{\mathbb{R}} = \mathbb{R} 
\otimes_{\mathbb{Z}} N$. We denote by $M$ the dual lattice of $N$ and 
the polar cone $\check{\sigma}$ of $\sigma$ in $M_{\mathbb{R}} = \mathbb{R} 
\otimes_{\mathbb{Z}} M$ by 
\[
\check{\sigma} = \left\{v \in M_{\mathbb{R}}; \ 
\ \left\langle u,v\right\rangle \geq 0 \ \ \text{for any} \ \ u \in 
\sigma\right\},
\] where $\langle \cdot , \cdot \rangle$ is the usual inner product in $\R^d$.
Then the dimension of $\check{\sigma}$ is $d$ and we obtain 
a semigroup $S_{\sigma}:= \check{\sigma} \cap M$.

\begin{definition}
	A $d$-dimensional affine toric variety $X_{\sigma}$ is defined by the 
	spectrum of $\mathbb{C}[S_{\sigma}]$,\ie 
	$X=\rm{Spec}(\mathbb{C}[S_{\sigma}])$.
\end{definition}

The algebraic torus $T = {\rm{Spec}}(\mathbb{C}[M])\cong 
(\mathbb{C}^{*})^d$ acts naturally on $X_{\sigma}$ and the $T$-orbits in 
$X_{\sigma}$ are indexed by the faces $\Delta $ of 
$\check{\sigma}$ ($\Delta \prec \check{\sigma}$). We denote by 
$\mathbb{L}(\Delta)$ the smallest linear 
subspace of $M_{\mathbb{R}}$ containing $\Delta$. For a face $\Delta$ of 
$\check{\sigma}$, denote by $T_{\Delta}$ the $T$-orbit in 
$\rm{Spec}(\mathbb{C}[M \cap \mathbb{L}(\Delta)])$ which corresponds to 
$\Delta$. We observe that the 
$d$-dimensional affine toric varieties are 
	exactly those $d$-dimensional affine, normal varieties admitting a 
	$(\mathbb{C}^{*})^d$-action with an open, dense orbit homeomorphic to 
	$(\mathbb{C}^{*})^d$. Moreover, each $T$-orbit $T_{\Delta}$ is 
	homeomorphic to $(\mathbb{C}^{*})^{r}$, where $r$ is the dimension of 
	$\mathbb{L}(\Delta)$.

Therefore we obtain a decomposition $X_{\sigma} = \bigsqcup_{\Delta 
\prec \check{\sigma}} T_{\Delta}$ into $T$-orbits, which are homeomorphic to 
algebraic torus $(\mathbb{C}^{*})^{r}$. Due to this fact, and also the 
informations coming from the combinatorial residing in these varieties, many 
questions that were originally studied for functions defined on 
$\mathbb{C}^d$ 
can be extended to functions defined on toric varieties.

Consider $f: X_{\sigma} \to \mathbb{C}$ a polynomial function on 
$X_{\sigma}$,\ie a function that corresponds to an element $f=\sum_{v\in 
S_{\sigma}}{a_v \cdot v}$ of $\mathbb{C}[S_{\sigma}]$, where $a_v \in 
\mathbb{C}$.

\begin{definition}
	Let $f=\sum_{v\in S_{\sigma}}{a_v \cdot v}$ be a polynomial function on 
	$X_{\sigma}$.
	
	\item (a) The set $\left\{v \in S_{\sigma}; \ \ a_v \neq 0 \right\} 
	\subset S_{\sigma}$ is called the support of $f$ and we denote it by 
	$\supp f$;

	\item (b) The Newton polygon $\Gamma_{+}(f)$ of $f$ is the convex hull of 
	\[
	\bigcup_{v \in \supp f} (v+ \check{\sigma}) \subset \check{\sigma}.
	\]	
\end{definition}                

Now let us fix a function $f \in \mathbb{C}[S_{\sigma}]$ such that $0 \notin 
\supp f$,\ie $f: X_{\sigma} \to \mathbb{C}$ vanishes at the $T$-fixed point 
$0$. Considering $M(S_{\sigma})$ the $\mathbb{Z}$-sublattice of rank $d$ in 
$M$ generated by $S_{\sigma}$, we have that each element $v$ of $S_{\sigma} 
\subset M(S_{\sigma})$ is identified with a $\mathbb{Z}$-vector 
$v=(v_1,\dots,v_d)$. Then for $g = \sum_{v \in S_{\sigma}} b_v \cdot v \in 
\mathbb{C}[S_{\sigma}]$ we can associate a Laurent polynomial $L(g) = \sum_{v 
\in S_{\sigma}} b_v \cdot x^{v}$ on $T = (\mathbb{C}^*)^d$, where 
$x^{v}:=x_1^{v_1}\cdot x_2^{v_2} \ldots \; x_d^{v_d}$.

\begin{definition}\label{degenerate}
  We say that $f=\sum_{v\in S_{\sigma}}{a_v\cdot v} \in 
  \mathbb{C}[S_{\sigma}]$ is non-degenerate if for any compact face $\gamma$ of 
  $\Gamma_{+}(f)$ the complex hypersurface 
  \begin{equation*}
  \bigsetdef{x=(x_1,\dots,x_d) \in 
  (\mathbb{C}^*)^d}{L(f_{\gamma})(x) = 0}
  \end{equation*}
  in $(\mathbb{C}^*)^d$ 
  is smooth and reduced, where $f_{\gamma} := \sum_{v \in \gamma \cap 
  S_{\sigma}} a_v \cdot v$.
\end{definition}

We can also study non-degeneracy in case of complete intersections defined on 
$X_{\sigma}$. Let $f_1,f_2, \dots, f_k \in \mathbb{C}[S_{\sigma}]$ $(1 \leq k 
\leq d = \dim X_{\sigma})$ and consider the following subvarieties of 
$X_{\sigma}$: 
$$V:= \left\{f_1=\dots=f_{k-1} = f_k =0 \right\} \subset W:= 
\left\{f_1=\dots=f_{k-1} =0 \right\}.$$ Assume that $0 \in V$. For each face 
$\Delta \prec \check{\sigma}$ such that $\Gamma_{+}(f_k) 
\cap \Delta \neq \emptyset$, we set 
\begin{equation*}
I(\Delta)= \bigsetdef{ j=1,2,\dots, k-1}{\Gamma_{+}(f_j) \cap \Delta \neq 
\emptyset} \subset \left\{ 
1,2,\dots,k-1 \right\}
\end{equation*} 
and $m(\Delta)= \# I(\Delta) + 1$.

Let $\mathbb{L}(\Delta)$ and $M(S_{\sigma} \cap \Delta)$ be as before and 
$\mathbb{L}(\Delta)^{*}$ the dual vector space of $\mathbb{L}(\Delta)$. Then 
$M(S_{\sigma} \cap \Delta)^{*}$ is naturally identified with a subset of 
$\mathbb{L}(\Delta)^{*}$ and the polar cone $\check{\Delta}= \bigsetdef{ u 
\in 
\mathbb{L}(\Delta)^{*}}{ \left\langle u, v\right\rangle \geq 0 \ \ 
\textup{for any} \ \ v\in \Delta}$ of $\Delta$ in 
$\mathbb{L}(\Delta)^{*}$ is a rational polyhedral convex cone with respect to 
the lattice $M(S_{\sigma} \cap \Delta)^{*}$ in $\mathbb{L}(\Delta)^{*}$.

\begin{definition} 
	(i) For a function $f = \displaystyle{ \sum_{v\in \Gamma_{+}(f)}} a_v \cdot 
	v 
	\in \mathbb{C}[S_{\sigma}]$ on $X_{\sigma}$ and $u \in \check{\Delta}$, 
	we set $f|_{\Delta} = \sum_{v\in \Gamma_{+}(f) \cap \Delta} a_v \cdot v \in 
	\mathbb{C}[S_{\sigma} \cap \Delta]$ and 
	\begin{equation*}
	\Gamma(f|_{\Delta};u) = 
	\bigsetdef{v \in \Gamma_{+}(f)\cap \Delta}{\left\langle u, v\right\rangle 
	= {\rm min} \left\langle u,w\right\rangle, \ \ \text{for} \ \ w \in 
	\Gamma_{+}(f) \cap \Delta}.
	\end{equation*} 
	We call $\Gamma(f|_{\Delta};u)$ the 
	supporting face of $u$ in $\Gamma_{+}(f) \cap \Delta$.
	
	(ii) For $j \in I(\Delta) \cup \left\{k \right\}$ and $u \in 
	\check{\Delta}$, we define the $u$-part $f_{j}^{u} \in 
	\mathbb{C}[S_{\sigma} \cap \Delta]$ of $f_j$ by 
	\[
	f_{j}^{u} = 
	\displaystyle{\sum_{v\in \Gamma(f_{j}|_{\Delta};u)} a_v \cdot v} \in 
	\mathbb{C}[S_{\sigma} \cap \Delta],
	\] 
	where $f_j = 
	\displaystyle{\sum_{v\in \Gamma_{+}(f_j)} a_v \cdot v} \in 
	\mathbb{C}[S_{\sigma}]$.
\end{definition}

By taking a $\mathbb{Z}$-basis of $M(S_{\sigma})$ and identifying the 
$u$-parts $f_{j}^{u}$ with Laurent polynomials $L(f_{j}^{u})$ on $T = 
(\mathbb{C}^{*})^d$ as before, we have that the following definition does 
not depend on the choice of the $\mathbb{Z}$-basis of $M(S_{\sigma})$.

\begin{definition}\label{ICdegenerate}
	We say that $(f_1,\dots,f_k)$ is non-degenerate if for any face $\Delta 
	\prec \check{\sigma}$ such that $\Gamma_{+}(f_k) \cap \Delta \neq 
	\emptyset$ (including the case where $\Delta = \check{\sigma}$) and any 
	$u \in {\rm Int}(\check{\Delta}) \cap M(S_{\sigma} \cap \Delta)^{*}$ the 
	following two subvarieties of $(\mathbb{C}^{*})^d$ are non-degenerate 
	complete intersections 
	\[
	\bigsetdef{x \in (\mathbb{C}^{*})^d}   
	{L(f_{j}^{u})(x)=0, \  \forall  j \in I(\Delta) }; \
	\bigsetdef{x \in (\mathbb{C}^{*})^d}{L(f_{j}^{u})(x)=0,  \ \forall j \in 
	I(\Delta) \cup \left\{k\right\}},
	\]
\noindent {{{\it i.e.}}, if they are reduced smooth complete intersections varieties in the torus $(\mathbb{C}^{*})^{d}$.}
\end{definition}
\noindent This last definition is presented in \cite{MT1} and it is a 
generalization of 
the one in \cite{Oka}.

For these non-degenerate singularities, it is possible to describe 
their geometrical and topological properties using the combinatorics of the 
Newton polygon. This is 
done in \cite{MT1} using mixed volume as follows. For each face $\Delta \prec 
\check{\sigma}$ of $\check{\sigma}$ such that 
$\Gamma_{+}(f_k) \cap \Delta \neq \emptyset$, we set 
$$
f_{\Delta} = 
\Bl\displaystyle{\prod_{j\in I(\Delta)}} f_j\Br\cdot f_k \in 
\mathbb{C}[S_{\sigma}]
$$ 
and consider its Newton polygon $\Gamma_{+}(f_{\Delta}) = \left\{\sum_{j \in 
	I(\Delta)} \Gamma_{+}(f_j) \right\} + \Gamma_{+}(f_k) \subset 
\check{\sigma}$. Let $\gamma_{1}^{\Delta},\dots, 
\gamma_{\nu(\Delta)}^{\Delta}$ be the compact faces of 
$\Gamma_{+}(f_{\Delta}) \cap \Delta (\neq \emptyset)$ such that 
$\dim\gamma_{i}^{\Delta} = \dim\Delta -1$. Then for each $1\leq i \leq 
\nu(\Delta)$ there exists a unique primitive vector $u_{i}^{\Delta} \in {\rm 
	Int}(\check{\Delta}) \cap M(S_{\sigma} \cap \Delta)^{*}$ which takes its 
minimal value in $\Gamma_{+}(f_{\Delta}) \cap \Delta$ exactly on 
$\gamma_{i}^{\Delta}$.

For $j \in I(\Delta) \cup \left\{k \right\}$, set $\gamma(f_j)_{i}^{\Delta} 
:= \Gamma(f_{j}|_{\Delta};u_{i}^{\Delta})$ and $d_{i}^{\Delta}:= {\rm 
	min}_{w\in \Gamma_{+}(f_k) \cap \Delta} \left\langle 
u_{i}^{\Delta},w\right\rangle$. Note that we have $$\gamma_{i}^{\Delta} = 
\displaystyle{\sum_{j \in I(\Delta) \cup \left\{k \right\}}} \gamma 
(f_j)_{i}^{\Delta}$$ for any face $\Delta \prec \check{\sigma}$ such that 
$\Gamma_{+}(f_k) \cap \Delta \neq \emptyset$ and $1 \leq i \leq \nu(\Delta)$. 
For each face $\Delta \prec \check{\sigma}$ such that $\Gamma_{+}(f_k) \cap 
\Delta \neq \emptyset$, $\dim\Delta \geq {\rm m}(\Delta)$ and $1 \leq i 
\leq \nu(\Delta)$, we set $I(\Delta) \cup \left\{k\right\} = 
\left\{j_1,j_2,\dots,j_{{\rm m}(\Delta)-1},k=j_{{\rm m}(\Delta)} \right\}$ 
and $$K_{i}^{\Delta} := \sum_{{{ \alpha_1+\dots+\alpha_{{\rm m}(\Delta)}  =  
			\dim\Delta -1} \atop {\alpha_q \geq 1 \ \ \text{for} \ \ q\leq 
			{\rm 
				m}(\Delta)-1}} \atop {\alpha_{{\rm m}(\Delta)} \geq 0}}^{} {\rm 
	Vol}_{\mathbb{Z}}(\underbrace{\gamma(f_{j_1})_{i}^{\Delta},\dots,
	\gamma(f_{j_1})_{i}^{\Delta}}_{\alpha_1-\text{times}},\dots,\underbrace{
	\gamma(f_{j_{{\rm
				m}(\Delta)}})_{i}^{\Delta},\dots,\gamma(f_{j_{{\rm 
				m}(\Delta)}})_{i}^{\Delta}}_{\alpha_{{\rm 
				m}(\Delta)}-\text{times}}).$$ Here 
$${\rm Vol}_{\mathbb{Z}}(\underbrace{\gamma(f_{j_1})_{i}^{\Delta},\dots,
	\gamma(f_{j_1})_{i}^{\Delta}}_{\alpha_1-\text{times}},\dots,\underbrace{
	\gamma(f_{j_{{\rm
				m}(\Delta)}})_{i}^{\Delta},\dots,\gamma(f_{j_{{\rm 
				m}(\Delta)}})_{i}^{\Delta}}_{\alpha_{{\rm 
				m}(\Delta)}-\text{times}})$$ is the 
normalized $(\dim\Delta -1)$-dimensional mixed volume with respect to 
the lattice $M(S_{\sigma} \cap \Delta) \cap L(\gamma_{i}^{\Delta})$ (see 
Definition $2.6$, pg $205$ from \cite{GKZ}). For $\Delta$ such that 
$\dim\Delta -1 =0$, we set $$K_{i}^{\Delta} = {\rm 
	Vol}_{\mathbb{Z}}(\underbrace{\gamma(f_{k})_{i}^{\Delta},\dots,
	\gamma(f_{k})_{i}^{\Delta}}_{0-\text{times}})
:= 1 $$ (in this case $\gamma(f_k)_{i}^{\Delta}$ is a point). 

%%%%%%%%%%%%%%%%%%%%%%%%%%%%%%%%%%%%%%%%%%%
\section{The Brasselet number and torus actions} 
%%%%%%%%%%%%%%%%%%%%%%%%%%%%%%%%%%%%%%%%%%%

We present formulas for the computation of the Brasselet number of 
a function defined on a non-degenerate complete intersection contained in 
toric variety using Newton 
polygons.  As applications we establish several results concerning 
its invariance for families of non-degenerate complete 
intersections.

Let $X_{\sigma} \subset \mathbb{C}^n$ be a $d$-dimensional toric variety and 
$(f_1,\dots,f_k): (X_{\sigma},0) \to (\mathbb{C}^k,0)$ a non-degenerate 
complete intersection, with $1\leq k \leq d$. From now on we will denote by $g$ 
the complete intersection $(f_1,\dots,f_{k-1})$ and by $f$ the function $f_k$. 
Let $\mathcal{T}$ be the decomposition of $X_{\sigma} = 
\bigsqcup_{\Delta \prec \check{\sigma}} T_{\Delta}$  into $T$-orbits, and 
$\mathcal{T}_{g}$ the decomposition of $X_{\sigma}^{g} = 
\bigsqcup_{\Delta \prec \check{\sigma}} T_{\Delta} \cap X_{\sigma}^{g}$. Since 
$g$ is a non-degenerate complete intersection, $\mathcal{T}_g$ is a decomposition of $X_\sigma^g$ into smooth subvarieties \cite[Lemma 4.1]{MT1}.  We are interest in  the situation that 
$\mathcal{T}_g$ is a Whitney stratification. 

\begin{example}
	Let $S_{\sigma} = \mathbb{Z}_{+}^3$ and let $X_{\sigma} = \mathbb{C}^3$ 
	be the smooth $3$-dimensional toric variety. Consider 
	$(g,f):(X_{\sigma},0) \to (\mathbb{C}^2,0)$ a non-degenerate complete 
	intersection, where
					$$
				 			g(z_1,z_2,z_3) = z_2^2 - z_1^3 -z_1^2z_3^2.
					$$
		\begin{figure}[h!]\centering
		\includegraphics[scale=0.5]{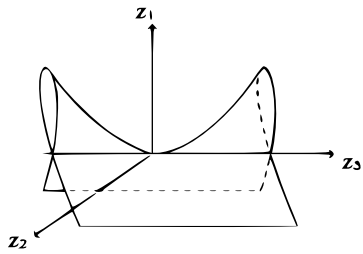}
		\caption{Real ilustration to 
			$X_{\sigma}^{g}$.}\label{fig:cuspidedeformation}
	\end{figure}

	The decomposition $\mathcal{V} = \{V_{0}=\{0\}, \ \ 					
	V_{1}=\{\text{axis} - z_3\} \setminus \{0\}, \ \ 	V_{2}  = 	
	(X_{\sigma}^{g})_{\rm{reg}}\}$ is a Whitney stratification of the 
	non-degenerate 	complete intersection $X_{\sigma}^{g}$, where 			
	$(X_{\sigma}^{g})_{\rm{reg}}$ is the regular part of  $X_{\sigma}^{g}$ 
	(see Figure \eqref{fig:cuspidedeformation}).	For a subset $I 
	\subset 		\{1,2,3\}$ we denote the face $\sum_{i \in I} 
	{\mathbb{R}_{+}}{e_{i}}^{*}$ of $\check{\sigma} = 
	(\mathbb{R}_{+}^{3})^{*}$ 		by $\Delta_{I}$ and 		by $\Delta_0$ 
	the face $\{0\}$, where ${e_i}^*$ 				denotes the 	
	elements of the canonical basis 
	of 																		
						$(\mathbb{R}^{3})^{*}$.
		 The Whitney stratification 				$\mathcal{T}$ of 
	$X_{\sigma}$ is given by the  following  $T$-orbits:	$T_{\Delta_{0}} 
	= 	\{0\}$, $T_{\Delta_{1}} = 								
	\mathbb{C}^{*} \times \{0\} \times \{0\}$, $T_{\Delta_{2}} = \{0\} 
	\times 								\mathbb{C}^{*} 		\times \{0\}$, 
	$T_{\Delta_{3}} = \{0\} \times \{0\} \times 	\mathbb{C}^{*}$, 
	$T_{\Delta_{12}} =		\mathbb{C}^{*} 
	\times 															
	\mathbb{C}^{*}
	 \times \{0\}$, $T_{\Delta_{13}} =	\mathbb{C}^{*} \times \{0\} 
	\times 			\mathbb{C}^{*}$, $T_{\Delta_{23}} = \{0\} 
	\times						\mathbb{C}^{*} \times \mathbb{C}^{*}$, 
	$T_{\Delta_{123}} = \mathbb{C}^{*} \times	\mathbb{C}^{*} 		\times 
	\mathbb{C}^{*}$. Now, observe that 				
	$T_{\Delta_{3}} 			\cap X_{\sigma}^{g} = T_{\Delta_{3}} = 
	V_{1}$, then the stratification 	
			$$
					\mathcal{T}_g = \left\{ T_{\Delta_{I}} 			\cap X_{\sigma}^{g}, \ \ \Delta_{0} \right\}_{\Delta_{I} \prec \check{\sigma}}
			$$
	is a Whitney stratification of $X_{\sigma}^{g}$.	
\end{example}

In Section \ref{Section-ToricSurface} we will provide others examples where 
the stratification $\mathcal{T}_{g}$ is Whitney (see \cite[Section 9]{Oka3} 
for more examples).
We observe that in general $\mathcal{T}_g$ is not a Whitney 
stratification see \cite[Example 9.5]{Oka3}.

\begin{theorem}\label{BrasseletIC}
	Let $X_{\sigma} \subset \mathbb{C}^n$ be a $d$-dimensional toric variety 
	and $(g,f): (X_{\sigma},0) \to (\mathbb{C}^k,0)$ a non-degenerate complete 
	intersection. If $\mathcal{T}_g$ is a Whitney stratification, then
	\begin{equation*}
	{\rm B}_{f,X_{\sigma}^{g}}(0)   =  \sum_{{\Gamma_{+} 
	(f) \cap \Delta \neq \emptyset} \atop {\dim \Delta \ \geq \ m(\Delta)}} 
	(-1)^{\dim \Delta \ - \ m(\Delta)} \left( \sum_{i=1}^{\nu(\Delta)} 
	d_{i}^{\Delta} \cdot K_{i}^{\Delta} \right)\cdot 
	\rm{Eu_{X_{\sigma}^{g}}}(T_{\Delta} \cap X_{\sigma}^g).
	\end{equation*}
\end{theorem}

\begin{proof}
	 Let $(X_\sigma^g)^f := X^g_\sigma \cap X^f_\sigma$ be the zero set of $f$ 
	 in $X^g_\sigma$. By \cite[Section 1]{Massey} or \cite[Section 
	 1.7]{GoMac}, there exists a refinement 
	 of $\mathcal{T}_g$, denoted by $\mathcal{T}_{(g,f)} = \{W_i\}$, such that 
	 $\mathcal{T}_{(g,f)}$ is 
	 adapted to $(X_\sigma^g)^f$ and it is a Whitney stratification. 
	 Moreover, this stratification satisfies the $(a_f)$-Thom condition (see 
	 \cite{P}), and then it 
	 is a good stratification of $X^g_\sigma$ with relation to $f$. By 
	 definition, the Brasselet number reads
	 \[
	 {\rm B}_{f,X_\sigma^g}(0) = \sum_{i}\chi \big(W_i\cap 
	 B_{\varepsilon}(0) 
	 \cap f^{-1}(\delta) \big) \cdot
	 {\rm Eu}_{X_\sigma^g}(W_i),
	 \]
	 where $0< \left| \delta \right| \ll \varepsilon \ll 1$.
	 
	 As $\mathcal{T}_{(g,f)}$ is a refinement of $\mathcal{T}_g$, each 
	 $T_\Delta\cap X^g_\sigma$ is a disjoint union of $W_i$, \ie 
	 \begin{equation*}
	 T_\Delta\cap X^g_\sigma = \bigsqcup_{W_i \cap T_\Delta\cap X^g_\sigma  
	 \not = \emptyset} W_{i},
	 \end{equation*}
	 and since $\mathcal{T}_g$ is a Whitney stratification,
	 \begin{equation*}
	 {\rm Eu}_{X_\sigma^g}(T_\Delta\cap X^g_\sigma) = {\rm 
	 Eu}_{X_\sigma^g}(W_i),
	 \end{equation*}
	 for any $W_i$ that meets $T_\Delta\cap X^g_\sigma$. Then
	 \begin{equation}\label{Eq-BrasseleNumber1}
	 {\rm B}_{f,X_\sigma^g}(0) = \sum_{\Delta}\chi \big(T_\Delta\cap X^g_\sigma 
	 \cap 
	 B_{\varepsilon} 
	 \cap f^{-1}(\delta) \big) \cdot
	 {\rm Eu}_{X_\sigma^g}(T_\Delta\cap X^g_\sigma).
	 \end{equation}
	 Note that $f:X_{\sigma} \to \mathbb{C}$ has the isolated 
	 stratified critical value $0 \in \mathbb{C}$ by 
	 \cite[Proposition 1.3]{Massey}. Thus the integers $\chi (T_\Delta\cap 
	 X^g_\sigma  \cap	 B_{\varepsilon}	 
	 \cap f^{-1}(\delta) \big)$ can be calculated by the nearby cycle functor 
	 ${\psi}_{f}$ \cite[Eq. (4.2)]{MT2},\ie
	 \begin{equation}\label{nearbycicle1}
	 \chi \big(	T_\Delta\cap X^g_\sigma  \cap	 B_{\varepsilon}	 \cap 
	 f^{-1}(\delta) \big)	= \chi(\psi_{f}(\mathbb{C}_{T_{\Delta} \cap 	
	 X_{\sigma}^{g}})_{0}),
	 \end{equation}
	 since $X_{\sigma}^{g}\cap B_\varepsilon \cap f^{-1}(\delta)  = 
	 \bigsqcup_{\Delta \prec 
	 \check{\sigma}} T_{\Delta} \cap B_\varepsilon\cap X_{\sigma}^{g} \cap 
	 f^{-1}(\delta)$ is a 
	 Whitney stratification of the Milnor fiber $X_{\sigma}^{g}\cap 
	 B_\varepsilon \cap 
	 f^{-1}(\delta)$. By the proof of \cite[Theorem 
	 3.12]{MT1} equations (52), (77), (78), we have that
	 \begin{equation}\label{nearbycicle2}
	 \chi(\psi_{f}(\mathbb{C}_{T_{\Delta} \cap X_{\sigma}^{g}})_{0}) =  
	(-1)^{\dim \Delta \ - \ m(\Delta)} \left( \sum_{i=1}^{\nu(\Delta)} 
	d_{i}^{\Delta} \cdot K_{i}^{\Delta} \right).
	 \end{equation}
	 To complete the proof, we observe that $f\equiv 0$ on $T_{\Delta}$ for any 
	 face $\Delta$ such that $\Gamma_{+}(f) \cap \Delta = \emptyset$, 
	 then we can neglect those faces. 
\end{proof}

If $f: X_{\sigma}^{g} \to \mathbb{C}$ has a stratified isolated
critical point, then Equation \eqref{EqBrasEul} holds. Altogether, we have: 

\begin{corollary}\label{BrasseletIC3}
	Let $X_{\sigma} \subset \mathbb{C}^n$ be a $d$-dimensional toric 
	variety and $(g,f): (X_{\sigma},0) \to (\mathbb{C}^k,0)$ a non-degenerate 
	complete intersection. If $\mathcal{T}_g$ is a Whitney 
	stratification and $f: X_{\sigma}^{g} \to \mathbb{C}$ 
	has an isolated  singularity at the origin, then
	\begin{equation*}	
	{\rm Eu}_{f,X_{\sigma}^{g}}(0)   =  {\rm Eu}_{X_{\sigma}^{g}}(0) - 
	\sum_{{\Gamma_{+} (f) \cap \Delta \neq \emptyset} \atop 
	{\dim\Delta \ \geq \ m(\Delta)}} (-1)^{\dim \Delta \ - \ m(\Delta)} 
	\left( \sum_{i=1}^{\nu(\Delta)} d_{i}^{\Delta} \cdot K_{i}^{\Delta} 
	\right)\cdot 
	\rm{Eu_{X_{\sigma}^{g}}}(T_{\Delta} \cap X_{\sigma}^g).
	\end{equation*}
\end{corollary}

\medskip

When $k=1$, using a similar argument used in Theorem \ref{BrasseletIC}, we 
obtain ${\rm B}_{f,X_{\sigma}}(0)$. In fact, using again \cite{GoMac,P} we can 
pass to a refinement of $\mathcal{T}$ to obtain a good stratification of 
$X_\sigma$ relative to $f$. Thus we can use the same argument as 
in the proof of Theorem \ref{BrasseletIC} to obtain  
\begin{equation*}
 {\rm B}_{f,X_\sigma}(0) = \sum_{\Delta}\chi \big(T_\Delta 
\cap 
B_{\varepsilon}(0) 
\cap f^{-1}(\delta) \big) \cdot
{\rm Eu}_{X_\sigma}(T_\Delta).
\end{equation*}
Now for each face $\Delta 
\prec 
\check{\sigma}$ such that 
$\Gamma_{+}(f) \cap \Delta \neq \emptyset$, let 
$\beta_{1}^{\Delta}$, $\beta_{2}^{\Delta}$, 
$\dots,\beta_{\mu(\Delta)}^{\Delta}$ be 
the compact faces of $\Gamma_{+}(f) \cap \Delta$ such that 
$\rm{dim}\beta_{i}^{\Delta} = \dim\Delta -1$. Let $\Gamma_{i}^{\Delta}$ 
be the convex hull of $\beta_{i}^{\Delta} \sqcup \left\{0\right\}$ in 
$\mathbb{L}(\Delta)$ and consider the normalized $(\dim\Delta)$-dimensional 
volume of 
$\Gamma_{i}^{\Delta}$,
$\rm{Vol}_{\mathbb{Z}}(\Gamma_{i}^{\Delta}) \in \mathbb{Z}$,  with respect to 
the lattice $M(S_{\sigma} \cap 
\Delta)$. Here $M(S_{\sigma} \cap \Delta)$ denotes the sublattice of 
$M(S_{\sigma})$ generated by $S_{\sigma} \cap \Delta$. Then we have the 
following result.

\begin{proposition}\label{propositionBrasselet}
	Assume that $f = \sum_{v\in S_{\sigma}} a_v \cdot v \in 
	\mathbb{C}[S_{\sigma}]$ is non-degenerate. Then $${\rm 
	B}_{f,X_{\sigma}}(0)= \sum_{\Gamma_{+} (f) \cap \Delta \neq 
	\emptyset}(-1)^{\dim\Delta - 1} \big( 
	\sum_{i=1}^{\mu(\Delta)}{\rm{Vol_{\mathbb{Z}}}}(\Gamma_{i}^{\Delta}) 
	\big)\cdot \rm{Eu_{X_{\sigma}}}(T_{\Delta}).$$  
\end{proposition}

\begin{proof} 
	For each $(\dim\Delta -1)$-compact face $\beta_{i}^{\Delta}$ of 
	$\Gamma_{+}(f) \cap \Delta \neq \emptyset$, with $1\leq i \leq 
	\mu(\Delta)$ we have 
	\begin{equation*}
	K_{i}^{\Delta} := {\rm 
	Vol}_{\mathbb{Z}}(\underbrace{\beta_{i}^{\Delta},\dots,\beta_{i}^{\Delta}}_{{({\rm
	 dim}\Delta -1})-\text{times}}),
	\end{equation*}
 	and from 
	\cite[Proposition 2.7]{GKZ} we now that 	
	\begin{equation*}
	{\rm 
	Vol}_{\mathbb{Z}}(\underbrace{\beta_{i}^{\Delta},\dots,
	\beta_{i}^{\Delta}}_{{(\dim\Delta \ -\ 1})-\text{times}}) = {\rm 
	Vol}_{\mathbb{Z}}(\beta_{i}^{\Delta}).
	\end{equation*}
	Therefore, the result follows 
	from the fact that 
	\begin{equation}
	{\rm{Vol_{\mathbb{Z}}}}(\Gamma_{i}^{\Delta}) = 
	d_{i}^{\Delta}\cdot K_{i}^{\Delta},
	\end{equation}
	for $1\leq i \leq \mu(\Delta)$.
\end{proof}

We will apply Theorem \ref{BrasseletIC} in order to show that the Brasselet 
number is invariant for some families of complete intersections.  We need 
some new concepts and notations.

%Following, we present some results concerning about the invariance of the Brasselet number.
\LHchange{A familia \'e para todo $t\in\C$? } \TDchange{Ok}
\begin{definition}
A deformation of a map germ $f : (X, 0) \to (\mathbb{C}^{k}, 0)$ is another
map germ $F : (\mathbb{C}\times X) \to (\mathbb{C}^{k}, 0)$ such that $F(0, 
x) = f(x)$, for all $x \in X$. 
\end{definition}
We assume
that $F$ is origin preserving, that is, $F(t, 0) = 0$ for all $t \in \mathbb{C}$, so we have a $1$-parameter family of
map germs $f_t : (X, 0) \to (\mathbb{C}^{k}, 0)$ given by $f_t(x) = F(t, x)$. 
Moreover, associated to the family $f_t : (X, 0) \to (\mathbb{C}^{k}, 0)$ we 
have the family $X^{f_t} = X \cap f_{t}^{-1}(0)$ of subvarieties of $X$.
In the particular case of a polynomial function $f : (X, 0) \to (\mathbb{C}, 
0)$, any polynomial deformation $f_t$ can be written as:
			\begin{equation}\label{family1}
					f_t(x) = f(x)+ \sum_{i=1}^{r} \theta_{i}(t) \cdot  h_{i}(x)
			\end{equation}
for some polynomials $h_{i}: (X,0) \to (\mathbb{C},0)$ and $\theta_{i}:(\mathbb{C},0) \to (\mathbb{C},0)$, where $\theta_{i}(0)=0$, for all $i=1,\dots,r$.

Suppose that $f: X_{\sigma} \to \mathbb{C}$ is a polynomial function  
defined on a toric variety $X_{\sigma}$. Consider a family as in 
\Eqref{family1} with $\Gamma_{+}(h_i) \subset \Gamma_{+}(f)$, for all 
$i=1,\dots,r$. Let
$\gamma_{i_1}^{\Delta},\gamma_{i_2}^{\Delta},\dots,
\gamma_{i_{\nu(\Delta)}}^{\Delta}$ 
and
$\beta_{1}^{\Delta},\beta_{2}^{\Delta},\dots,\beta_{\mu(\Delta)}^{\Delta}$ 
be the compact faces of $\Gamma_{+}(h_i) \cap 
\Delta$ and of $\Gamma_{+}(f) \cap \Delta$, respectively, such that 
$\dim\gamma_{l}^{\Delta}=\dim\beta_{j}^{\Delta} = \dim\Delta -1$.
If for each face $\Delta \prec 
\check{\sigma}$, with $\Gamma_{+}(h_i) \cap \Delta \neq 
\emptyset$, 
we have that $\gamma_{i_l}^{\Delta} \cap \beta_{j}^{\Delta} = \emptyset$, for 
all 
$l=1,\dots,i_{\nu(\Delta)}$ and $j=1,\dots,\mu(\Delta)$, then
\begin{equation*}
\Gamma_+(f_t) = \Gamma_+(f), \ \ \text{for all} \ \ t\in \mathbb{C}.
\end{equation*}
In this case, we fix the notation  			
\begin{equation}\label{condition}
	\Gamma_{+}(h_i) \subsetneqq \Gamma_{+}(f),
\end{equation}
for all $i=1,\dots,r$.
\medskip

In the sequel, we present some applications of Theorem \ref{BrasseletIC}.

\begin{corollary}\label{BrasseletIC2}
	Let $X_{\sigma} \subset \mathbb{C}^n$ be a $d$-dimensional toric variety.
	Suppose that
	$(g,f_t): (X_{\sigma},0) \to (\mathbb{C}^k,0)$ is a 
	non-de\-ge\-ne\-ra\-te 
	complete intersection for each $t\in\mathbb{C}$ with
	\[
	f_t(x) = f(x)+\sum_{i=1}^{r} \theta_{i}(t) \cdot h_{i}(x)
	\]
	being a polynomial function on $X_{\sigma}$ with $h_i$ satisfying the 
	condition \eqref{condition} for all $i=1,\dots,r$. If 
	$\mathcal{T}_g$ is a Whitney 
	stratification and $(g,f_t)$ is a family of
	non-degenerate 					complete	intersections, then 
	${\rm B}_{f_t,X_{\sigma}^g}(0)$ is constant for the family. 
\end{corollary}

\begin{proof}
	As we have already noted, for each face $\Delta \prec \check{\sigma}$ 
	such that $\Gamma_{+}(f) \cap \Delta \neq \emptyset$, 
	the Newton polygon $\Gamma_{+}(f_{\Delta})$ of the function 
	\begin{equation*}
	f_{\Delta} = \Bl{\prod_{j\in I(\Delta)}} f_j\Br\cdot f \in 
	\mathbb{C}[S_{\sigma}]
	\end{equation*} 
	is $ \left\{\sum_{j \in I(\Delta)} 
	\Gamma_{+}(f_j) \right\} + \Gamma_{+}(f) \subset \check{\sigma}$. 
	Therefore, $\Gamma_{+}(f_\Delta) = \Gamma_{+}(({f_t})_\Delta)$, for all 
	$t \in \mathbb{C}$, where 
	\begin{equation*}
	(f_{t})_{\Delta} = \Bl\displaystyle{\prod_{j\in I(\Delta)}} f_j\Br\cdot 
	f_t,
	\end{equation*} since 
	$\Gamma_{+}(f) = \Gamma_{+}(f_t)$. Then the result follows by Theorem 
	\ref{BrasseletIC}.
\end{proof}

Roughly speaking the Brasselet number depends only on the 
monomials of smallest degree in each variable.

\medskip

Given $f: (X_{\sigma},0) \to (\mathbb{C},0)$ a non-degenerate function, as 
before, we denote by $\mathcal{T}_{f}$ the decomposition 
$$X_{\sigma}^{f}= \bigsqcup_{\Delta \prec \check{\sigma}} T_{\Delta} \cap 
X_{\sigma}^{f}$$ of $X_{\sigma}^{f}$, and by $\mathcal{T}^{f}$ the 
decomposition $\left\{ T_{\Delta} \cap X_{\sigma}^{f}, \ \  T_{\Delta} 
\setminus X_{\sigma}^{f}, \ \ \left\{0\right\} 					
\right\}_{\Delta \prec \check{\sigma}}$ of $X_{\sigma}$. If $\mathcal{T}_{f}$ 
is a Whitney stratification of $X_{\sigma}^{f}$, then $\mathcal{T}^{f} $ is a 
Whitney stratification of $X_{\sigma}$ which is adapted to $X_{\sigma}^{f}$. 
In fact, this follows from the fact that $\mathcal{T}$ is a Whitney 
stratification of $X_{\sigma}$. Consequently, $\mathcal{T}^f$ is 
a 
good 		
stratification of $X_{\sigma}$ relative to $f$.

\medskip

Let $g$ and $f$ be non-degenerate polynomial functions on $X_{\sigma}$. In 
general we have no way to relate 
the local Euler obstructions  ${\rm Eu}_{X_{\sigma}^{g}}(T_{\Delta} \cap 
X_{\sigma}^{g})$ 
to  ${\rm Eu}_{X_{\sigma}}(T_{\Delta})$. However, if we assume 
the 
additional hypothesis that $g$ has an isolated critical point at $0$ both in 
$X_{\sigma}$ and in $X_{\sigma}^f$ (in stratified sense), the 
following result holds.

\begin{theorem}\label{BrasseletIC4}
	Let $X_{\sigma} \subset \mathbb{C}^n$ be a $d$-dimensional toric variety 
	and $(g,f): (X_{\sigma},0) \to (\mathbb{C}^2,0)$ a non-degenerate complete intersection. Suppose that $\mathcal{T}^{f}$ is a good stratification   		of $X_{\sigma}$ relative to $f$ and that $g$ is prepolar with respect to $\mathcal{T}^{f}$, then
	\begin{equation*}
	{\rm B}_{f,X_{\sigma}^{g}}(0)   =  \sum_{{\Gamma_{+} 
	(f) \cap \Delta \neq \emptyset} \atop {\dim \Delta \ \geq \ 2}} 
	(-1)^{\dim\Delta \ - \ {2}} \left( \sum_{i=1}^{\nu(\Delta)} 
	d_{i}^{\Delta} \cdot K_{i}^{\Delta} \right)\cdot 
	\rm{Eu_{X_{\sigma}}}(T_{\Delta}). 
	\end{equation*}
\end{theorem}

\begin{proof} Consider $\mathcal{T}_{g} = \left\{ T_{\Delta} \cap 
X_{\sigma}^{g} \right\}_{\Delta \prec \check{\sigma}}$ the stratification 
of 	
$X_{\sigma}^{g}$. Using the same argument as in the proof of Theorem 
\ref{BrasseletIC} we can 
	pass to a refinement $\mathcal{T}_{(g,f)}$ and we obtain a good 
	stratification of 
	$X_{\sigma}^{g}$ relative to $f$. Then 
	\begin{equation*}
				{\rm B}_{f,X_{\sigma}^{g}}(0)   =  \sum_{\Gamma_{+}(f)\cap 
				\Delta \neq \emptyset}\chi \big(W_{\Delta}\cap B_{\varepsilon} 
				\cap 										f^{-1}(\delta) 
				\big) 
				\cdot
				{\rm Eu}_{X_{\sigma}^{g}}(W_{\Delta}),
				\end{equation*}
	where $W_{\Delta}$ are the strata of $\mathcal{T}_{(g,f)}$ which are not 
	contained in $\left\{f = 
	0 \right\}$, and $0 < \left|\delta\right| \ll \varepsilon \ll 1$. Moreover,
	for $\Delta \prec \check{\sigma}$, we have ${\rm 
	Eu}_{X_{\sigma}}(T_{\Delta}) = {\rm Eu}_{X_{\sigma}^{g}}(W_{\Delta})$, 
	since $X_{\sigma}^g$ intersects the strata of $\mathcal{T}^f$ transversally 
	(see \cite[Proposition $IV$. $4.1.1$]{Dubson}). Hence,
		\begin{equation*}
		\begin{aligned}
		&   {\rm B}_{f,X_{\sigma}^{g}}(0)  & =  &    \displaystyle { \sum_{\Gamma_{+}(f)\cap \Delta \neq \emptyset} }\chi \big(X_{\sigma}^{g} \cap 						T_{\Delta} \cap 
		B_{\varepsilon} \cap f^{-1}(\delta) \big) \cdot
		{\rm Eu}_{X_{\sigma}}(T_{\Delta}).  &
		\end{aligned}
		\end{equation*} 
	Finally, as $\Gamma_{+}(g) \cap \Delta \neq \emptyset$ for any 
	face $0\precneqq \Delta \prec \check{\sigma}$, then $m(\Delta) = 2$, for 
	all face $\Delta \prec \check{\sigma}$ such that $\Gamma_{+}(f) 
	\cap \Delta \neq \emptyset$. Since $g$ is prepolar, 
	\begin{equation*}
	X_{\sigma}^{g}\cap 
	B_\varepsilon \cap f^{-1}(\delta)  = 
	\bigsqcup_{\Delta \prec 
		\check{\sigma}} T_{\Delta} \cap B_\varepsilon\cap X_{\sigma}^{g} \cap 
	f^{-1}(\delta)
	\end{equation*} 
	is a 
	Whitney stratification of the Milnor fiber $X_{\sigma}^{g}\cap 
	B_\varepsilon \cap 
	f^{-1}(\delta)$. Therefore, we obtain the result by \Eqref{nearbycicle1} 
	and \eqref{nearbycicle2}.
\end{proof}

Therefore, if $\mathcal{T}^f$ is a good stratification of $X_{\sigma}$ relative to $f$ and $g$ is prepolar with respect to $\mathcal{T}^f$, we can obtain a more general version of Corollary \ref{BrasseletIC2}, since we can relate 
the local Euler obstructions  ${\rm Eu}_{X_{\sigma}^{g}}(T_{\Delta} \cap 
X_{\sigma}^{g})$ to the 
Euler local obstructions ${\rm Eu}_{X_{\sigma}}(T_{\Delta})$.

\begin{corollary}\label{BrasseletIC8}
	Let $X_{\sigma} \subset \mathbb{C}^n$ be a $d$-dimensional toric variety 
	and $(g,f): (X_{\sigma},0) \to (\mathbb{C}^2,0)$ a non-degenerate 
	complete intersection. Suppose that 
	\begin{equation*}
	\Bl g_s(x),f_t(x) \Br = \Bl g(x) + 
	\sum_{i=1}^{m}\xi_{i}(s)\cdot l_i(x), f(x)+\sum_{j=1}^{r} 
	\theta_{j}(t)\cdot h_j(x) \Br	 
	\end{equation*}
	is a family of non-degenerate complete 
	intersections with $l_i$ and $h_j$ satisfying the condition 
	\eqref{condition} for all $i=1,\dots,m$ and $j=1,\dots,r$. If 
	$\mathcal{T}^{f_t}$ is a good stratification   		of $X_{\sigma}$ 
	relative to $f_t$ and $g_s$ is prepolar with respect to 
	$\mathcal{T}^{f_t}$ at the origin, for all $s,t \in \mathbb{C}$,
	then, ${\rm B}_{f_t,X_{\sigma}^{g_s}}(0)$ is constant for all 
	$t,s \in \mathbb{C}$.
\end{corollary}

\begin{proof}
  Since $\Gamma_{+}(l_i) \subset \Gamma_{+}(g) \ \ \text{and} \ \ 
  \Gamma_{+}(h_j) \subset \Gamma_{+}(f)$, for each face $\Delta \prec 
  \check{\sigma}$   such that $\Gamma_{+}(f) \cap \Delta \neq \emptyset$, the 
  Newton polygon $\Gamma_{+}(f_{\Delta})$ of the function $$f_{\Delta} = 
  \Bl      	   \displaystyle{\prod_{\Gamma_{+}(f) \cap \Delta \neq \emptyset}} 
  g \Br \cdot f \in \mathbb{C}[S_{\sigma}]$$ is equals to   
  $\Gamma_{+}(({f_{t})}^{s}_{\Delta} )$, where $${(f_{t})}^{s}_{\Delta} = \Bl 
  \displaystyle{\prod_{\Gamma_{+}(f_t) \cap \Delta \neq \emptyset}} g_s \Br 
  \cdot f_t \in \mathbb{C}[S_{\sigma}].$$ Then by \Eqref{nearbycicle1} and 
  \eqref{nearbycicle2}, we can conclude 
  that the Euler characteristic 
 	$$
  	\chi \Bl X_{\sigma}^{g_s} \cap T_{\Delta} \cap 
	B_{\varepsilon} \cap f_{t}^{-1}(\delta) \Br
	$$ 
	is constant for all $s,t \in \mathbb{C}$. Moreover, as $g_s$ is prepolar 
	with respect to $\mathcal{T}^{f_t}$, then
	we can proceed exactly in the same way as in Theorem 
	\ref{BrasseletIC4} to obtain
	$$
	{\rm Eu}_{X_{\sigma}}(T_{\Delta}) = {\rm 
	Eu}_{X_{\sigma}^{g_s}}(T_{\Delta} \cap X_{\sigma}^{g_s}),
	$$ 
	and this concludes the proof.
\end{proof}

As a consequence of Proposition 
\ref{propositionBrasselet} and Theorem \ref{BrasseletIC4}, we have that if 
$\mathcal{T}^f$ 
is a good stratification of $X_{\sigma}$ relative to $f$ and if $g: 
X_{\sigma} \to 
\mathbb{C}$ is prepolar with respect to $\mathcal{T}^f$, we can express the 
number of stratified Morse critical points on the stratum of maximum dimension 
appearing in a morsefication of $g : X_{\sigma} \cap f^{-1}(\delta) \cap 
B_{\varepsilon} \to \mathbb{C}$ in terms of volumes of convex polytopes. 
More precisely, by Theorem \ref{Le Greuel}, we have 

\begin{equation*}
\begin{array}{lllll}

& (-1)^{d-1}n_d  &  = & \displaystyle{ \sum_{{\Gamma_{+} (f) \cap \Delta \neq 
\emptyset}}}(-1)^{\dim\Delta - 1} \big( 
\sum_{i=1}^{\mu(\Delta)}{\rm{Vol_{\mathbb{Z}}}}(\Gamma_{i}^{\Delta}) 
\big)\cdot 
\rm{Eu_{X_{\sigma}}}(T_{\Delta})  & \\

&   & - & \displaystyle \sum_{{\Gamma_{+} (f) \cap \Delta \neq \emptyset} 
\atop {\dim\Delta \geq 2}} (-1)^{\dim\Delta - {2}} \big( 
\sum_{i=1}^{\nu(\Delta)} d_{i}^{\Delta} \cdot K_{i}^{\Delta} \big)\cdot 
\rm{Eu_{X_{\sigma}}}(T_{\Delta}) & \\

\end{array}
\end{equation*}

\noindent where $n_d$ is the number of stratified Morse critical points on the top 
stratum $T_{\Delta_d} \cap f^{-1}(\delta)\cap B_{\varepsilon}$ appearing in a 
morsefication of $g : X_{\sigma} \cap f^{-1}(\delta) \cap B_{\varepsilon} \to 
\mathbb{C}$.
Therefore, if $f_t(x)= f(x)+\sum_{j=1}^{r} \theta_j(t) \cdot h_j(x)$ is a 
family of non-degenerate polynomial functions on $X_{\sigma}$ and if 
$(g_s,f_t): (X_{\sigma},0) \to (\mathbb{C}^2,0)$ is a family of non-degenerate 
complete intersections which satisfy the same hypotheses as Corollary 
\ref{BrasseletIC8}, then 							
	$$
	(-1)^{d-1}n_d = {\rm 
	B}_{f_t,X_{\sigma}}(0) - {\rm B}_{f_t,{X_{\sigma}^{g_s}}}(0)
	$$ 
	\noindent is constant for all $s,t \in \mathbb{C}$. More precisely, we 
	can state the following result.

\begin{corollary}\label{BrasseletIC5}
	Let $X_{\sigma} \subset \mathbb{C}^n$ be a $d$-dimensional toric variety 
	and $(g,f):(X_{\sigma},0) \to (\mathbb{C}^2,0)$ a non-degenerate complete 
	intersection. Suppose that 
\begin{equation*}
	\Bl g_s(x),f_t(x) \Br = \Bl g(x) + 
	\sum_{i=1}^{m}\xi_{i}(s)\cdot l_i(x), f(x)+\sum_{j=1}^{r} 
	\theta_{j}(t)\cdot h_j(x) \Br
\end{equation*}
	is a family of non-degenerate complete 
	intersections with $l_i$ and $h_j$ satisfying the condition 
	\eqref{condition} for all $i=1,\dots,m$ and $j=1,\dots,r$. If  
	$\mathcal{T}^{f_t}$ is a good stratification of $X_{\sigma}$ 
	relative to $f_t$ and $g_s$ is prepolar with respect to 
	$\mathcal{T}^{f_t}$ at the origin, for all $s,t \in \mathbb{C}$, 
	then 
	$(-1)^{d-1}n_d$ is constant for all $s, t \in \mathbb{C}$.
\end{corollary}

We will present some applications of the results presented before in the next 
section and some examples in Section \ref{Section-ToricSurface}.

%%%%%%%%%%%%%%%%%%%%%%%%%%%%%%%%%%%%%%%%%%%%%%%%%%%%%%%%%%%%%%%%%%%%%%%
%%%%%%%%%%%%%%%%%%%%%%%%%%%%%%%%%%%%%%%%%%%%%%%%%%%%%%%%%%%%%%%%%%%%%%%
%%%%%%%%%%%%%%%%%%%%%%%%%%%%%%%%%%%%%%%%%%%%%%%%%%%%%%%%%%%%%%%%%%%%%%%

%%%%%%%%%%%%%%%%%%%%%%%%%%%%%%%%%%%
\section{The local Euler obstruction and Bruce-Roberts' Milnor number}
%%%%%%%%%%%%%%%%%%%%%%%%%%%%%%%%%%%
\TDchange{GSV index vai para a se\c c\~ao 5 junto com a parte final}

\LHchange{A introducao da secao esta estranha, talvez colocar na ordem em que 
\'e a secao.}

\TDchange{+ -}

In this section, we derive sufficient 
conditions to obtain the invariance of the local Euler obstruction for families 
of 
complete intersections  which are 
contained in $X_{\sigma} \subset \C^n$. As an application, we study the 
invariance of the Bruce-Roberts' Milnor number for families of functions 
defined on hypersurfaces.

%%%%%%%%%%%%%%%%%%%%%%%%%%%%%%%%%%%%%%%%%%%%%%%%%%%%%%%%%%
\subsection{Local Euler obstruction of non-degenerate complete intersections}
%%%%%%%%%%%%%%%%%%%%%%%%%%%%%%%%%%%%%%%%%%%%%%%%%%%%%%%%%%%%%
 
As observed in \cite[Remark $2.5$]{BG} the local Euler obstruction is not a 
topological 
invariant. However, for non-degenerate complete intersections we have the following result.

\begin{theorem}\label{TheoBrasseletIC}
	Let  $X_{\sigma} \subset \mathbb{C}^n$ be a $d$-dimensional toric variety 
	and $k$ a positive integer in $\{2,\dots,d\}$. Consider 
	$g=(f_1,\dots,f_{k-1}): 
	(X_{\sigma},0) \to	(\mathbb{C}^{k-1},0)$ a non-degenerate complete 
	intersection, such that $\mathcal{T}_g$ is a Whitney stratification. 
	Suppose that
	\begin{equation*}
	g_s(x) = 		\Bl f_1(x) + \sum_{{i_1}=1}^{m_1} \theta_{i_1}(s) 
			\cdot h_{i_1}(x),\dots,f_{k-1}(x) + 
			\sum_{{i_{k-1}}=1}^{m_{k-1}} 					
			\theta_{i_{k-1}}(s) 
			\cdot h_{i_{k-1}}(x) \Br
	\end{equation*}
	is a family of non-degenerate complete intersections, such that 
	$\mathcal{T}_{g_s}$ is a Whitney stratification and $h_{i_p}$ satisfies 
	the condition \eqref{condition}
	for all $p \in \left\{1,\dots,k-1 \right\}$ and $i_p \in \left\{1, \dots, 
	m_p\right\}$. 
	If $L: \mathbb{C}^n 
	\to 	\mathbb{C}$ is a linear form  which is generic with respect to 
	$X_{\sigma}^{g_s}$ 
	for all $s \in \mathbb{C}$ and $(g_s,L)$ is a non-degenerate complete 
	intersection, then ${\rm Eu}_{X_{\sigma}^{g_s}}(0)$ is 
	invariant for the family $\{g_s\}_{s\in\C}$. 
\end{theorem}

\begin{proof}
		As $L: \mathbb{C}^n \to 	\mathbb{C}$ is a linear form  which is generic with respect to 
		$X_{\sigma}^{g_s}$, for all $s \in \mathbb{C}$,
		\begin{equation*}
		{\rm Eu}_{X_{\sigma}^{g_s}}(0) = {\rm B}_{L,X_{\sigma}^{g_{s}}}(0).
		\end{equation*}
		Moreover, $T_\Delta\cap X^{g_{s}}_\sigma  \cap	 B_{\varepsilon}	 
		\cap L^{-1}(\delta) $ induces a Whitney
		stratification on the Milnor fiber $X^{g_{s}}_\sigma  \cap	 
		B_{\varepsilon}	 \cap L^{-1}(\delta)$. Therefore, the result 
		follows from Theorem \ref{BrasseletIC}, and from the fact that for all 
		face 
		$\Delta 
		\neq 			\left\{0\right\}$ of $\check{\sigma}$, the Newton 
		polygon $\Gamma_{+}(L_{\Delta})$ of the function 
		$$
		L_{\Delta} = \Bl \displaystyle{\prod_{j\in I(\Delta)}} f_j 
		\Br \cdot L \in \mathbb{C}[S_{\sigma}]
		$$
is equal to $\Gamma_{+}({L_{\Delta}^{s}})$. Here, 
		$$
		{L_{\Delta}^{s}} = \Bl \displaystyle{\prod_{p\in I(\Delta)}} f_p + 
		\sum_{{i_p}=1}^{m_p} \theta_{i_p} \cdot h_{i_p} \Br \cdot  L 	
		\in 									\mathbb{C}[S_{\sigma}].
		$$
%	\noindent since $h_{i_p}$ satisfies the condition \eqref{condition}
%   	for all $p \in \left\{1,\dots,k-1 \right\}$ and \linebreak $i_p \in 
%\left\{1, \dots, 
%	  m_p\right\}$. 
\end{proof}

\begin{corollary}\label{ObstrucaoEulerIC}
	Let $S_{\sigma} = \mathbb{Z}_{+}^n$ and let $X_{\sigma} = \mathbb{C}^n$ 
	be 
	the smooth $n$-dimensional toric variety. Consider
	$g=(f_1,\dots,f_{k-1}): 					\mathbb{C}^n \to 
	\mathbb{C}^{k-1}$ a non-degenerate complete intersection with an isolated 
	singularity at $0$, where $2\leq k \leq n$. 
	Suppose that
			$$
			g_s(x) = 		\Bl f_1(x) + \sum_{{i_1}=1}^{m_1} \theta_{i_1}(s) 
			\cdot h_{i_1}(x),\dots,f_{k-1}(x) + 
			\sum_{{i_{k-1}}=1}^{m_{k-1}} 					\theta_{i_{k-1}}(s) 
			\cdot h_{i_{k-1}}(x) \Br
			$$ 
	is a family of non-degenerate complete intersections with an isolated 
	singularity at $0$, where $h_{i_p}$ satisfies the condition 
	\eqref{condition}
	for all $p \in \left\{1,\dots,k-1 \right\}$ and $i_p \in \left\{1, \dots, 
	m_p\right\}$. If $L: \mathbb{C}^n 
	\to 	\mathbb{C}$ is a linear form  which is generic with respect to 
	$X_{\sigma}^{g_s}$, 
	for all $s \in \mathbb{C}$ and $(g_s,L)$ is a non-degenerate complete 
	intersection, then ${\rm Eu}_{X_{\sigma}^{g_s}}(0)$ is 
	invariant for the family $\{g_s\}_{s\in\C}$.

\end{corollary}

\begin{proof}
		As $g_s$ is a family of complete intersections with an isolated 
		singularity at $0$, the decomposition $\mathcal{T}_{g_{s}}$ of 
		$X_{\sigma}^{g_{s}} = 
		\bigsqcup_{\Delta \prec \check{\sigma}} T_{\Delta} \cap 
		X_{\sigma}^{g_{s}}$ is a Whitney stratification, once $\mathcal{T}$ 
		is a Whitney stratification of $X_{\sigma}$. Therefore, the result 
		follows from Theorem \ref{TheoBrasseletIC}.
\end{proof}

With the same assumptions as in Theorem \ref{TheoBrasseletIC}, suppose that $f_t(x) = 
f_k(x)+\sum_{{i_k}=1}^{m_k} \theta_{i_k}(t)\cdot h_{i_k}(x)$ is a family of 
polynomial functions such that $(g_s,f_t)$ is a family of non-degenerate 
complete 
intersections where $h_{i_k}$ satisfies the condition \eqref{condition},\ie
	\begin{equation*}
		\Gamma_{+}(h_{i_k}) \subsetneqq 
		\Gamma_{+}(f_k), \ \ \text{for all} \ \ i_k=1,\dots,m_k.
	\end{equation*}
Moreover, assume that $f_t: X_{\sigma}^{g_s} \to \mathbb{C}$ has a stratified 
isolated critical point at $0$. For each face $\Delta \prec \check{\sigma}$ 
satisfying $\Gamma_{+}(f_k) \cap \Delta \neq \emptyset$, the Newton polygon 
$\Gamma_{+}(f_{\Delta})$ of the function 
	$$
		f_{\Delta} = \big(\displaystyle{\prod_{j\in I(\Delta)}} f_j    
		\big)\cdot f_k \in \mathbb{C}[S_{\sigma}]
	$$
\noindent is equals to $\Gamma_{+}((f_{t})^{s}_{\Delta} )$, where 
\TDchange{****  conferir as trocas de indices e $\xi$ por $\theta$}
	$$
		(f_{t})^{s}_{\Delta} = \Bl \displaystyle{\prod_{p\in I(\Delta)}} f_p 
		+ 
		\sum_{{i_p}=1}^{m_p} \theta_{i_p} \cdot h_{i_p} \Br \cdot \Bl f_t 
		= 						f_k+\sum_{i_{k}=1}^{m_k} \theta_{i_k}\cdot 
		h_{i_k} \Br \in \mathbb{C}[S_{\sigma}].
	$$
\noindent By Theorem \ref{BrasseletIC} we conclude that the 
Euler characteristic of the Milnor fiber of $f_t: X_{\sigma}^{g_s} \to 
\mathbb{C}$ is invariant for all $s,t\in\C$. Therefore, ${\rm Eu}_{f_t, 
X_{\sigma}^{g_s}}(0)$ is invariant for the family.

%%%%%%%%%%%%%%%%%%%%%%%%%%%%%%%%%%%%%%%%%%%%%%%%%%
\subsection{Bruce-Roberts' Milnor number}
%%%%%%%%%%%%%%%%%%%%%%%%%%%%%%%%%%%%%%%%%%%%%%%%%

In \cite{BR}, Bruce and Roberts introduced a Milnor number for 
functions germs on singular varieties. 

Let $X$ be a 
sufficiently small representative of the germ $(X,0)$ and let $I(X)$ denote 
the ideal in $\mathcal{O}_{n,0}$ consisting of the germs of functions 
vanishing on $X$. We say that two germs $f$ and $g$ in $\mathcal{O}_{n,0}$ 
are $\mathcal{R}_X$- equivalent if there exists a germ of diffeomorphism 
$\phi: (\mathbb{C}^n,0)\rightarrow (\mathbb{C}^n,0)$ such that $\phi(X)=X$ 
and $f\circ\phi=g$. Let $\theta_n$ denote the $\mathcal{O}_{n,0}$- module of 
germs of vector fields 
on $(\mathbb{C}^n,0).$ Each vector field $\xi\in\theta_n$ can be seen as a 
derivation $\xi:\mathcal{O}_{n,0}\rightarrow\mathcal{O}_{n,0}$. We denote by 
$\theta_X$ those vector fields that are tangent to $X$,\ie 
$$
\theta_X:=\bigsetdef{\xi\in\theta_n}{dg(\xi)=\xi g\in I(X), \forall g\in 
I(X)}.
$$

\begin{definition}
	Let $f$ be a function in $\mathcal{O}_{n,0}$ \LHchange{Acho que aqui est\'a faltando um be 
	alguma coisa} and let $df(\theta_X)$ be the ideal $\{\xi f: 
	\xi\in\theta_X\}$ in $\mathcal{O}_{n,0}$. The number 
	$$\mu_{BR}(X,f)=\dim_{\mathbb{C}}\frac{\mathcal{O}_{n,0}}{df(\theta_X)}$$
	is called the Bruce-Roberts number of $f$ with respect to $X$.
\end{definition}

We refer to \cite{BR} for more details about $\mu_{BR}(X,f)$. 
In particular, $\mu_{BR}(X,f)$ is finite
if and only if $f$ is $\mathcal{R}_X$-finitely determined.

An interesting open problem is to know whether the Bruce-Roberts number is a 
topological invariant or not. In \cite[Corollary $5.19$]{Grulha2, Grulha1} 
Grulha 
gave a partial answer. The author proved that, if $(X, 
0)$ is a hypersurface whose logarithmic characteristic variety ${\rm LC}(X)$ 
\cite[Definition $1.13$]{BR}\LHchange{colocar onde est\'a a definicao no artigo} \TDchange{ok}, is 
Cohen-Macaulay and if 
$f_t$ is a $C^0$- $R_X$-trivial deformation of $f$, then $\mu_{BR}(f_t,X)$ is 
constant. 

For any hypersurface $X$ the problem of ${\rm LC}(X)$ being Cohen-Macaulay 
remains open. When $X$ is a weighted homogeneous hypersurface with an isolated 
singularity, ${\rm LC}(X)$ is Cohen-Macaulay by \cite[Theorem $4.2$]{BBT2}.

Let us recall that $\mu(f)$ denotes the Milnor number \cite{M1} of a 
germ of an analytic function $f : (\mathbb{C}^n,0) \to (\mathbb{C}, 0)$ with 
an isolated critical point at the origin and it is defined as 
\[
\mu(f) = \dim_{\mathbb{C}} 
\frac{\mathcal{O}_{n,0}}{J (f )},
\] where $\mathcal{O}_{n,0}$ is the ring of germs of 
analytic functions at the origin, and $J(f )$ is the Jacobian ideal of $f$.

Using \cite[Theorem 4.2]{BBT2}, and assuming that $L$ satisfies the same 
hypotheses as in Corollary \ref{ObstrucaoEulerIC} we prove the following result.

\begin{proposition}\label{BruceRoberts}
	Let $S_{\sigma} = \mathbb{Z}_{+}^n$ and let $X_{\sigma} = \mathbb{C}^n$ 
	be the smooth $n$-dimensional toric variety. Consider 
	$(g,f):(X_{\sigma},0) \to (\mathbb{C}^2,0)$ a non-degenerate complete 
	intersection, and 
			$$
				\Bl g_s(x),f_t(x) \Br = \Bl g(x) + 
				\sum_{i=1}^{m}\xi_{i}(s)\cdot l_i(x), f(x)+\sum_{j=1}^{r} 
				\theta_{j}(t)\cdot h_j(x) \Br
			$$
	\noindent a family of non-degenerate complete intersections with $h_j$ 
	and $l_i$ satisfying the	condition \eqref{condition}. Suppose 
	that, 
	for all $s,t \in \mathbb{C}$,  $X_{\sigma}^{g_s} \subset \mathbb{C}^n$ is 
	a weighted homogeneous hypersurface with 	an		isolated singularity 
	at 
	the origin. If $f_t: \mathbb{C}^n \to \mathbb{C}$ is a polynomial 
	function 
	with an isolated singularity at the origin such that $f_t: 
	X_{\sigma}^{g_s} \to \mathbb{C}$ has also a stratified isolated 
	singularity at the origin, then $\mu_{BR}(f_t,X_{\sigma}^{g_s})$ is 
	constant to 		all $s,t \in \mathbb{C}$.
\end{proposition}
\begin{proof}
	From \cite[Corollary 2.38]{BrunaTese} we have 
				$$
					\mu_{BR}(f_t,X_{\sigma}^{g_s}) = \mu(f_t) + {\rm Eu}_{X_{\sigma}^{g_s}}(0) + (-1)^{n-1}({\rm Eu}_{f_t,{X_{\sigma}}^{g_s}}(0) + 1).
				$$
By the hypothesis $X_\sigma = \C^n$, then $L(g^{\mu})(x) = g_\mu(x)$, for all 
$\mu \in {\rm Int}(\check{\Delta}) \cap M(S_{\sigma} \cap \Delta)^{*}$.
From definitions \ref{degenerate} and \ref{ICdegenerate}, we can conclude that	
$f_t:\mathbb{C}^n \to \mathbb{C}$ is a family of non-degenerate polynomial 
functions. Furthermore,
	$$
	\Gamma_{+}(f) = \Gamma_{+}(f_t), \ \ \text{for all} \ \ t \in \mathbb{C},
	$$ 
	since $h_j$ satisfies the	condition \eqref{condition}. Then,	by 
	\cite[Corollary 3.5]{MT1} 
	$$
	\chi(f^{-1}(\delta) \cap B_{\varepsilon}) = \chi(f_{t}^{-1}(\delta) \cap 
	B_{\varepsilon}), \ \ \text{for all} \ \ t \in \mathbb{C},
	$$
	where $0 < \left| \delta 
	\right| \ll \varepsilon \ll 1 $. Consequently, $\mu(f_t)$ is constant, once 
	$\chi(f^{-1}(\delta) \cap 
	B_{\varepsilon})= 1 + (-1)^{n-1} \mu(f)$.  	Therefore the result follows 
	from 
	Corollary \ref{ObstrucaoEulerIC} and the remark that follows Corollary 
	\ref{ObstrucaoEulerIC}. 
\end{proof}

We observe that this result is a kind of generalization of\cite[Theorem 
$3.6$]{BBT3}.

%%%%%%%%%%%%%%%%%%%%%%%%%%%%%%%%%%%%%%%%%%%%%%%%%%%%%%%%%%%%%%%%%%%%%%%%%%%
%%%%%%%%%%%%%%%%%%%%%%%%%%%%%%%%%%%%%%%%%%%%%%%%%%%%%%%%%%%%%%%%%%%%%%%%%%%

%%%%%%%%%%%%%%%%%%%%%%%%%%%%%%%%%%%%%%%%%%%%%%%%%%%%%%%%%%%%%%%%
\section{Toric surfaces}\label{Section-ToricSurface}
%%%%%%%%%%%%%%%%%%%%%%%%%%%%%%%%%%%%%%%%%%%%%%%%%%%%%%%%%%%%%%%%%

Let $f$ be a polynomial function defined on a $2$-dimensional toric variety 
$X_{\sigma} \subset \mathbb{C}^{n}$. In this section, we present a characterization 
of the polynomial functions $g: X_{\sigma} \to \mathbb{C}$ which are 
prepolar with respect to $\mathcal{T}^f$ at the origin. Using this characterization and the results of the last sections we 
present some examples of computation of the Brasselet number ${\rm 
B}_{f,X_{\sigma}}$, for  
a class of toric surfaces $X_{\sigma}$ that are also determinantal.

Let us remember that a strongly convex cone in $\mathbb{R}^2$ has the 
following normal form.

\begin{proposition}[\cite{F}]\label{proposition1}
	Let $\sigma \subset \mathbb{R}^2$ be a strongly convex cone, then 
	$\sigma$ is isomorphic to the cone generated by the vectors $v_1 = pe_1 - 
	qe_2$ and $v_2 = e_2$, for some integers $p,q \in \mathbb{Z}_{>0}$ such 
	that $0<q<p$ and $p,q$ are coprime.
\end{proposition}

Given a cone $\sigma \subset \mathbb{R}^2$, Riemenschneider \cite{Rie1,Rie2} 
proved
 that the binomials which generate the ideal $I_{\sigma}$ are 
given by the {\it{quasiminors}} of a {\it{quasimatrix}}, where $X_{\sigma} = 
V(I_{\sigma})$. In the following we recall the definition of 
{\it{quasimatrix}}.

\begin{definition} 
	Given $A_{i}, B_{i}, C_{l,l+1} \in \mathbb{C}$ with $i=1,\dots,n$ and 
	$l=1,\dots,n-1$, a quasimatrix with entries $A_{i}, B_{i}, C_{l,l+1}$ is 
	written as
	\begin{equation*}
	\begin{matrix} A=
	\begin{pmatrix}
	A_1   & \ \ \ \  & A_2  & \cdots &    A_{n-1} & \ \ \ \ &  A_n \\
	
	B_1  & \ \ \ \ & B_2 &  \cdots   &   B_{n-1}  & \ \ \ \ & B_n \\
	
	& C_{1,2} & & \cdots   & &  C_{n-1,n} &
	\end{pmatrix}.
	\end{matrix}
	\end{equation*}
	
	\noindent The quasiminors of the quasimatrix $A$ are defined by 
	$$
	A_i \cdot B_j - B_i\cdot (C_{i,i+1}\cdot C_{i+1,i_2} \cdots  
	C_{j-1,j})\cdot A_j
	$$ \LHchange{Adicionei pontos de multiplicacao na equacao, estao certos?}\TDchange{ok}
	for $1 \leq i < j \leq n$.
\end{definition}

Given $\sigma \subset \mathbb{R}^2$ generated by $v_1 = p e_{1} - q e_{2}$ 
and $v_{2} = e_{2}$, with $p$ and $q$ as above, let us consider the 
Hirzebruch-Jung continued fraction $$\frac{p}{p-q} = a_2 - \frac{1}{a_3 - 
\frac{1}{\dots - \frac{1}{a_{n-1}}}} = [[a_2, a_3, \dots,a_{n-1}]] $$ where 
the integers $a_2, \dots, a_{n-1}$ satisfies $a_i \geq 2$, for 
$i=2,\dots,n-1$. By \cite{Rie2} we have: 
\LHchange{completar...}\TDchange{Esse acho que n\~ao da pra colocar pq eu encontrei essa configura\c c\~ao em uma traducao na verdade}

\begin{proposition}\label{proposition2}
	The ideal $I_{\sigma}$ is generated by the quasiminors of the quasimatrix
	\begin{equation*}
	\begin{matrix}
	\begin{pmatrix}
	z_1  & \ \ \ \ \ \ & z_2 & \ \ \ \ \ \ & z_3 & \cdots &    z_{n-2} & \ \ 
	\ \ \ \ & z_{n-1} \\
	
	z_2 & \ \ \ \ \ \ & z_3 & \ \ \ \ \ \ & z_4 &  \cdots   &   z_{n-1} & \ \ 
	\ \ \ \ &  z_{n} \\
	
	&	z_2^{a_2 - 2} & & z_3^{a_3 - 2} & &  \cdots   & & z_{n-1}^{a_{n-1}-2} 
	\\
	\end{pmatrix},
	\end{matrix}
	\end{equation*}
	where the $a_i$ are given by the Hirzebruch-Jung continued fraction of 
	$\frac{p}{p-q}$. Moreover, this set of generators is minimal.
\end{proposition}

Then, if $a_i = 2$ for $i=3,\dots, n-2$, we have that $X_{\sigma}$ is a 
determinantal surface \cite{GGR, BBT, MC}, in particular if the minimal dimension of 
embedding of $X_{\sigma}$ is $4$,\ie if $$\frac{p}{p-q} = a_2 - 
\frac{1}{a_3} $$ then $X_{\sigma}$ is always determinantal and the ideal 
$I_{\sigma}$ is generated by the $2 \times 2$ minors of the matrix

\begin{equation*}
\begin{matrix}
\begin{pmatrix}
z_1  & \ \ & z_2 & \ \  & z_3^{a_3 -1}  \\

z_2^{a_2 -1} & \ \  & z_3 &  \ \ & z_4  \\
\end{pmatrix}.
\end{matrix}
\end{equation*}

We will consider $\sigma$ as in Proposition \ref{proposition1}. Take 
$a_2,\dots, a_{n-1}$ the integers coming from the Hirzebruch-Jung 
continued fraction of $\frac{p}{p-q}$, \LHchange{onde entram os a's?} \TDchange{ok} we will 
denote by  $$\mu_1 = 
(\mu_1^{1},\mu_1^{2})=(1,0), \ \ \mu_2 = (\mu_2^{1},\mu_2^{2})=(1,1), \ \ 
\mu_{i+1}^{j}=a_i \cdot \mu_{i}^{j} - \mu_{i-1}^{j},$$ the minimal set of 
generators of $S_{\sigma}$, with $i=2,\dots,n-1$; $j=1,2$. We note that it is 
possible to show that $\mu_n =(\mu_n^{1},\mu_n^{2})=(q,p)$ (see \cite{Rie1, 
Rie2}). Thus $\varphi: (\mathbb{C}^*)^2 \times X_{\sigma} \to X_{\sigma}$ 
given by 
$$\varphi((t_1,t_2),(z_1,\dots,z_n)) = (t_1 \cdot z_1,t_1 \cdot t_2 \cdot 
z_2,t_1^{\mu_3^{1}} \cdot t_2^{\mu_3^{2}} \cdot z_3,\dots, t_1^{q} \cdot 
t_2^{p} \cdot z_n)$$
is an action of $(\mathbb{C}^{*})^2$ in $X_{\sigma}$. Each orbit of $\varphi$ 
is an embedding of a $d$-dimensional torus, $0\leq d 
\leq 2$, in $X_{\sigma}$. The action $\varphi$ has $4$ orbits, that are
\begin{equation*}
\begin{array}{lrl}
T_{\Delta_0} & = &\{(0,\dots,0) \} \\

T_{\Delta_1} & = & \bigsetdef{(t_1,0,\dots,0)}{t_1 \in \mathbb{C}^* } \cong 
\mathbb{C}^*
\\

T_{\Delta_2} & = &\bigsetdef{ (0,\dots,0,t_1^{q} \cdot t_2^{p})}{t_1,t_2 \in 
\mathbb{C}^*} \cong \mathbb{C}^* \\

T_{\Delta_3} & = & \bigsetdef{
(t_1,t_1 \cdot t_2,t_1^{\mu_3^{1}} \cdot t_2^{\mu_3^{2}},\dots,t_1^{q} \cdot 
t_2^{p})}{t_1,t_2 \in \mathbb{C}^*} \cong (\mathbb{C}^*)^2
\end{array}.
\end{equation*}

\noindent Moreover, as in Section $3$, $$X_{\sigma} = \bigsqcup_{\Delta_i 
\prec \check{\sigma}} T_{\Delta_i},$$ with $i=0,1,2,3$, is a decomposition of 
$X_{\sigma}$ in strata satisfying the Whitney conditions. 

The toric surfaces obtained in Proposition \ref{proposition2} are normal 
toric surfaces, then they are smooth or they have isolated singularity at the 
origin. Therefore, if $f =\sum_{v\in S_{\sigma}} a_v \cdot v \in 
\mathbb{C}[S_{\sigma}]$ is a 
	non-degenerate polynomial function on $X_{\sigma}$, then 
	$$
	\mathcal{T}_f=\bigsetdef{T_{\Delta_i} \cap X_{\sigma}^{f}}{i=0,1,2,3},
	$$
	is a Whitney stratification of $X_{\sigma}^{f}$, since $T_{\Delta_{1}}$ and $T_{\Delta_{2}}$ are smooth subvarieties of $\overline{T}_{\Delta_{3}} = 		X_{\sigma}$ which satisfy $\overline{T}_{\Delta_{1}} \cap \overline{T}_{\Delta_{2}} = \{ (0,\dots,0)\}$. As a consequence,
		$$
			\mathcal{T}^{f}=\bigsetdef{T_{\Delta_i} \setminus X_{\sigma}^{f}, 
			\ \ 
			T_{\Delta_i} \cap X_{\sigma}^{f}, \ \ \left\{0 \right\}}{ 
			i=0,1,2,3},
		$$
is a good stratification of $X_{\sigma}$ relative to $f$. 
	
\medskip

Next, we characterize the polynomial functions 
which have a stratified isolated singularity at the origin.

\begin{lemma}\label{lemma2}
	Let $\sigma \subset \mathbb{R}^2$ be a strongly convex cone and 
	$\mathcal{T}$ the Whitney stratification of $X_{\sigma} \subset 
	\mathbb{C}^n$ whose the strata are $T_{\Delta_0}$, $T_{\Delta_1}$, 
	$T_{\Delta_2}$ and $T_{\Delta_3}$. Then, a non-degenerate polynomial 
	function $g: (X_{\sigma},0) \to (\mathbb{C},0)$ has an isolated singularity 
	at the origin (in the stratified sense) if, 
	and only if, $$g(z_1,\dots,z_n) = c_1z_1^{p_1} + h(z_1,\dots,z_n) + 
	c_nz_n^{p_n},$$ where $h$ is a polynomial function on $X_{\sigma}$, 
	$c_1,c_n \in \mathbb{C}^{*}$ and $p_1,p_n \in \Z^*_+$.
\end{lemma}

\begin{proof} 
	Let us write $g$ as follows $$g(z_1,\dots,z_{n}) = 
	\displaystyle{\sum_{l=1}^{m} c_lz_1^{p_1^l} z_2^{p_2^l} \dots 
	z_{n}^{p_{n}^l}},$$ where $l=1,\dots,m$, $p_i^{l} \in \Z^*_+$ 
	and $c_l \in \mathbb{C}$.
	
	Suppose that $g$ has a stratified isolated singularity at the origin $0 
	\in \mathbb{C}^n$, with respect to the stratification $\mathcal{T}$, then there must be $l_1,l_n \in \left\{1,\dots,m 
	\right\}$ such that
	
	\begin{equation*}
	\begin{array}{lrl}
	c_{l_1} \in \mathbb{C}^*, \ \ &  p_1^{l_1} \neq 0 \ \ \text{and}  \ \ & 
	p_{i}^{l_1} = 0, \ \ \text{for} \ \ i \in \left\{2,\dots,n\right\} \\
	
	c_{l_n} \in \mathbb{C}^*, \ \ &  p_n^{l_n} \neq 0 \ \ \text{and}  \ \ & 
	p_{i}^{l_n} = 0, \ \ \text{for} \ \ i \in \left\{1,\dots,n-1\right\}
	\end{array},
	\end{equation*}
	
	\noindent otherwise $T_{\Delta_1}, T_{\Delta_2} \subset 
	\Sigma_{\mathcal{T}} g$, since
	\begin{equation*}
	T_{\Delta_1}  =  \bigsetdef{(t_1,0,\dots,0)}{t_1 \in \mathbb{C}^*},\ 
	\textup{and}\ 
	T_{\Delta_2}  =  \bigsetdef{ (0,\dots,0,t_1^{q} \cdot t_2^{p})}{t_1,t_2 
	\in \mathbb{C}^*}. 
	\end{equation*}
	In other words, $g$ must contain monomials of the 
	form $c_1z_1^{p_1}$ and $c_nz_n^{p_n}$. 
	
	Now, suppose that $g$ has the form mentioned above, then  $\Gamma_{+}(g) 
	\cap \Delta_1 \neq \emptyset$, $\Gamma_{+}(g) \cap \Delta_2\neq 		
	\emptyset$ and  
	$\Gamma_{+}(g) \cap \Delta_3 \neq \emptyset$. By the proof of 
	\cite[Lemma 4.1]{MT1} we can conclude 
	that 									$g|_{T_{\Delta_{1}}}: 
	T_{\Delta_{1}} \to \mathbb{C}$ and $g|_{T_{\Delta_{2}}}: T_{\Delta_{2}} 
	\to \mathbb{C}$ are non-degenerate polynomial 					
	functions. Moreover,
	\begin{equation*}
		\begin{array}{lrl}

		T_{\Delta_1} & = & \bigsetdef{(t_1,0,\dots,0)}{t_1 \in \mathbb{C}^* } \cong 
		\mathbb{C}^*,
		\\

		T_{\Delta_2} & = &\bigsetdef{ (0,\dots,0,t_1^{q} \cdot t_2^{p})}{t_1,t_2 \in 
		\mathbb{C}^*} \cong \mathbb{C}^*, \\

		T_{\Delta_3} & = & \bigsetdef{
		(t_1,t_1 \cdot t_2,t_1^{\mu_3^{1}} \cdot t_2^{\mu_3^{2}},\dots,t_1^{q} \cdot 
		t_2^{p})}{t_1,t_2 \in \mathbb{C}^*} \cong (\mathbb{C}^*)^2.
		\end{array}
	\end{equation*}
	 Then by \cite[Lemma 78]{Oka} there exists an $\varepsilon>0$ such 
	 that $g|_{T_{\Delta_{1}}}$, $g|_{T_{\Delta_{2}}}$ and 
	 $g|_{T_{\Delta_{3}}}$	
	have no singularities in ${T_{\Delta_{1}}} \cap B_{\varepsilon}$, 
	${T_{\Delta_{2}}} \cap B_{\varepsilon}$ and ${T_{\Delta_{3}}} 
	\cap 									B_{\varepsilon}$, respectively. 
\end{proof}

As a consequence of Lemma \ref{lemma2} we obtain information about the 
singular set of $g$ by
just looking at its Newton polygon $\Gamma_{+}(g)$. More precisely, a 
non-degenerate polynomial function $g: (X_{\sigma},0) \to (\mathbb{C},0)$ has 
an isolated 
singularity at the origin (in the stratified sense) if, and only if, $\Gamma_{+}(g)$ intersects 
$\Delta_1$ and $\Delta_2$, exactly in 
the same way as the classic case,\ie in the case where $X_{\sigma} = 
\mathbb{C}^2$.

\begin{proposition}
Let $(g,f): (X_{\sigma},0) \to (\mathbb{C}^2,0)$ be a non-degenerate complete intersection, such that $f$ and $g$ have no irreducible components in common. The polynomial function $g$ is prepolar with respect to $\mathcal{T}^f$ if, and only if, 
$$g(z_1,\dots,z_n) = c_1z_1^{p_1} + h(z_1,\dots,z_n) + c_nz_n^{p_n},$$ where 
$h$ is a polynomial function on $X_{\sigma}$, $c_1,c_n \in \mathbb{C}^{*}$ 
and $p_1,p_n \in \Z_{+}^*$.
\end{proposition}

\begin{proof}
	Consider the good stratification 
	\begin{equation*}
	\mathcal{T}^f = \bigsetdef{T_{\Delta_i} \setminus X_{\sigma}^{f}, \ \ 
	T_{\Delta_i} \cap X_{\sigma}^{f}, \ \ \left\{0 \right\}}{i=0,1,2,3},
	\end{equation*}
	of $X_{\sigma}$ relative to $f$. The sets $T_{\Delta_1}$ and $T_{\Delta_2}$ 
	are given by
		\begin{equation*}
		T_{\Delta_1}  =  \bigsetdef{(t_1,0,\dots,0)}{t_1 \in \mathbb{C}^*}\ 
		\textup{and}\ 
		T_{\Delta_2}  =  \bigsetdef{ (0,\dots,0,t_1^{q} \cdot t_2^{p})}{t_1,t_2 
		\in \mathbb{C}^*}. 
		\end{equation*}
	From the fact that $f$ is a non-degenerate polynomial function, there are 
	only two possibilities for each stratum $T_{\Delta_i} 
	\cap X_{\sigma}^{f}$. Either $T_{\Delta_i} \cap X_{\sigma}^{f}$ is a finite 
	set or $T_{\Delta_i} \cap X_{\sigma}^{f} = T_{\Delta_i}$. Therefore, if $g$ 
	is prepolar with respect 
	to 									$\mathcal{T}^f$, then $\Gamma_{+}(g) 
	\cap \Delta_{1} \neq \emptyset$ and $\Gamma_{+}(g) \cap \Delta_{2} 
	\neq \emptyset$, otherwise 						$T_{\Delta_i} \cap 
	X_{\sigma}^{f} \subset \Sigma_{\mathcal{T}^f} g$ or $T_{\Delta_i} \setminus 
	X_{\sigma}^{f} \subset \Sigma_{\mathcal{T}^f} g$, for 			$i=1,2$.

	Now, suppose that $g$ have the form mentioned above, then the result 
	follows 
	from Lemma \ref{lemma2} and from the fact that, as $f$ and $g$ have no 
	irreducible components in common,\ie $X_{\sigma}^{g} \cap X_{\sigma}^{f}$ 
	is a finite set.
\end{proof}

\begin{example}
	Let $\sigma \subset \mathbb{R}^2$ be the cone ge\-ne\-ra\-ted by the 
	vectors $v_1 = e_2$ and $v_2 = ne_1 - e_2$. The toric surface associated 
	to $\sigma$ is $X_{\sigma} = V(I_{\sigma}) \subset \mathbb{C}^{n+1}$, 
	where $I_{\sigma}$ is the ideal generated by the $2 \times 2$ minors of 
	the matrix
	\begin{equation*}
	\begin{matrix}
	\begin{pmatrix}
	z_1 & z_2 & z_{3} & \dots & z_{n-1} & z_n\\
	z_{2} & z_3 & z_4 & \dots & z_n     & z_{n+1}
	\end{pmatrix}
	\end{matrix},
	\end{equation*}
	i.e., $X_{\sigma}$ is a determinantal 
	surface with codimension $n-1$. Consider $f: X_{\sigma} \to \mathbb{C}$ 
	the function given by 
	$f(z_1,\dots,z_{n+1})=z_1^d + z_{n+1}^{d} + tg(z_1,\dots,z_{n+1})$, where 
	$$
	g(z_1,\dots,z_{n+1}) = \displaystyle{\sum_{l=1}^{m} z_1^{p_1^l} 
	z_2^{p_2^l} \dots z_{n+1}^{p_{n+1}^l}}
	$$
	\noindent is a polynomial function on $X_{\sigma}$ satisfying $p_1^{l} + 
	p_2^{l} + \dots + p_{n+1}^{l} > d$  
	for every $l=1,\dots,m$. If $f$ is a 
	non-degenerate polynomial function then 
	$$
	{\rm B}_{f,X_{\sigma}}(0)= 2d -nd^2.
	$$
	\noindent Indeed, consider $h:X_{\sigma} 
	\to \mathbb{C}$ the function given by $h(z_1,\dots,z_{n+1}) = z_1^d + 
	z_{n+1}^d$. The Newton polygon $\Gamma_+(h)$ has an unique 
	$1$-dimensional compact face $\beta_1$, that is the straight line segment 
	connecting the points $(d,0)$ and $(d,nd)$ in $\check{\sigma}$. Using the same notation of Proposition \ref{propositionBrasselet}, we have that 
	$\Gamma_1^{\Delta_1}$ is the straight line segment connecting the points 
	$(0,0)$ and $(d,0)$, $\Gamma_1^{\Delta_2}$ is the straight line segment 
	connecting the points $(0,0)$ and $(d,nd)$ and $\Gamma_1^{\Delta_3}$ is 
	the triangle with vertices $(0,0)$, $(d,0)$ and $(d,nd)$. Therefore, 
	$$
	\rm{Vol_{\mathbb{Z}}}(\Gamma_{1}^{\Delta_1})=\rm{Vol_{\mathbb{Z}}}(\Gamma_{1}^{\Delta_2})
	 = d \ \ \text{and} \ \ 
	\rm{Vol_{\mathbb{Z}}}(\Gamma_{1}^{\Delta_3})=nd^2.
	$$
	By Proposition \ref{propositionBrasselet},
	$$
	{\rm 	B}_{h,X_{\sigma}}(0)=2d-nd^2,
	$$
	since $X_{\sigma}$ has an isolated singularity at the origin, and 
	consequently ${\rm 	Eu}_{X_{\sigma}}(T_{\Delta_{1}}) = 
	{\rm 													
	Eu}_{X_{\sigma}}(T_{\Delta_{2}})
	 = {\rm 	Eu}_{X_{\sigma}}(T_{\Delta_{3}}) = 1$.
	\noindent Now, $S_{\sigma}$ is the semigroup generated by 
	$$
	\left\{(1,0),(1,1),(1,2)\dots,(1,n) \right\},
	$$
	and then $\Gamma_+(g) \subset 
	\Gamma_+(h)$. Moreover, by Lemma	\ref{lemma2}, $f$ has an isolated 
	singularity at the origin, 
	thus 
	$$
	{\rm B}_{f,X_{\sigma}}(0)= {\rm Eu}_{X_{\sigma}}(0) - {\rm 
	Eu}_{f,X_{\sigma}}(0).
	$$
	\noindent However, Gonz\'alez-Sprinberg \cite{Gonzalez} proved that ${\rm 	Eu}_{X_{\sigma}}(0) = 3-(n+1)$. Hence, 
	$$
			{\rm Eu}_{f,X_{\sigma}}(0) = 3-(n+1)-2d+nd^2.
	$$
	Therefore, a morsefication of $f$ 		has $3-(n+1)-2d+nd^2$ Morse points 
	on the regular part of $X_{\sigma}$. 
\end{example}

\begin{example}\label{examplemorse} \TDchange{colocar uma familia em $g$ tambem e usar esse exemplo no final da secao 3 e tb no final da secao 5}
	Let $\sigma \subset \mathbb{R}^2$ be the cone ge\-ne\-ra\-ted by the 
	vectors $v_1 = e_2$ and $v_2 = 2e_1 - e_2$. The toric surface associated 
	to $\sigma$ is $X_{\sigma} = V(I_{\sigma}) \subset \mathbb{C}^3$, with 
	$I_{\sigma}$ the ideal generated by $z_1z_3 - z_2^{2}$. Consider $f: 
	X_{\sigma} \to \mathbb{C}$ the function given by $f(z_1,z_2,z_3)=z_2^2 - 
	z_1^3$, which is a non-degenerate polynomial function, whose singular 
	set is 
	$$
	\Sigma f = \bigsetdef{(0,0,z_3)}{z_3 \in \mathbb{C}} 
	\subset X_{\sigma}.
	$$
	Moreover, $\Gamma_+(f)$ has a unique $1$-dimensional 
	compact face $\beta_1$, which is the straight line segment connecting the 
	points $(3,0)$ and $(2,2)$ in $\check{\sigma}$. Thus, 
	$\Gamma_1^{\Delta_1}$ is the straight line segment connecting the points 
	$(0,0)$ and $(3,0)$, $\Gamma_1^{\Delta_3}$ is the triangle with vertices 
	$(0,0)$, $(3,0)$ and $(2,2)$, and $\Gamma_1^{\Delta_2} = \emptyset$. 
	Therefore, by Proposition \ref{propositionBrasselet}, 
	$$
			{\rm B}_{f,X_{\sigma}}(0)=3-6=-3
	$$
	since $X_{\sigma}$ has an isolated singularity at the origin, and 
	consequently 
	$$
				{\rm 	Eu}_{X_{\sigma}}(T_{\Delta_{1}}) = {\rm Eu}_{X_{\sigma}}(T_{\Delta_{2}}) = {\rm 	Eu}_{X_{\sigma}}(T_{\Delta_{3}}) = 1.
	$$
	Now let $g: X_{\sigma}	\to \mathbb{C}$ be the non-degenerate polynomial function given by 
	$$
				g(z_1,z_2,z_3)= z_1 - z_3^2,
	$$
	 which is prepolar 
	with respect to $\mathcal{T}^f$. Moreover, $(g,f)$ is a non-degenerated 
	complete intersection. The Newton polygon $\Gamma_{+}(g\cdot f)$ has two 
	$1$-dimensional 
	compact faces $\gamma_1$ and $\gamma_2$, which are the straight line 
	segment connecting the points $(4,0)$ and $(3,2)$ and the straight line 
	segment connecting the points $(3,2)$ and $(4,6)$, respectively. Thus, the 
	primitive vectors 
	$$u_{1}^{\Delta_3}, u_{2}^{\Delta_3}  \in {\rm 
	Int}(\check{\Delta_3}) \cap M(S_{\sigma} \cap \Delta_3)^{*}
	$$ 
	\noindent which take their minimal value in $\Gamma_{+}(g\cdot f) \cap 
	\Delta_3$ exactly on $\gamma_{1}$ 
	and $\gamma_2$, respectively, are $u_{1}^{\Delta_3}=(2,1)$ and 
	$u_{2}^{\Delta_3}=(4,-1)$. Let us observe that
	
	\begin{equation*}
	\begin{array}{lllll}
	\gamma(g)_{1}^{\Delta_3} & := & \Gamma(g|_{\Delta_3};u_{1}^{\Delta_3}) & 
	= & \left\{(1,0) \right\} \\
	
	\gamma(g)_{2}^{\Delta_3} & := & \Gamma(g|_{\Delta_3};u_{2}^{\Delta_3}) & 
	= & \alpha_1 \\
	
	d_{1}^{\Delta_3} & := & d_{2}^{\Delta_3} & := &	6 \\
	
	K_1^{\Delta_3} & = & K_2^{\Delta_3} & = & 1
	 
	\end{array},
	\end{equation*}
	
	\noindent where $\alpha_1$ is the $1$-dimensional compact face of 
	$\Gamma_{+}(g)$. Applying Theorem	\ref{BrasseletIC4}, we have 
	${\rm B}_{f,X_{\sigma}^{g}}(0)= 12$.
	Therefore, we obtain the following equality 
	$$
	{\rm B}_{f,X_{\sigma}}(0) - {\rm 
	B}_{f,X_{\sigma}^{g}}(0) = -3 - 12 =-15,
	$$
	which means that the number 
	of stratified Morse critical points on the top stratum $T_{\Delta_3} \cap 
	f^{-1}(\delta)\cap B_{\varepsilon}(0)$ appearing in a morsefication of $g : 
	X_{\sigma} \cap f^{-1}(\delta) \cap B_{\varepsilon}(0) \to \mathbb{C}$ is 
	$15$. Moreover, if we consider $h, l: X_{\sigma} \to \mathbb{C}$ the 
	polynomial functions given by 
				\begin{equation*}
					h(z_1,z_2,z_3)  =  -z_1^2 z_3^2, \qquad
					l(z_1,z_2,z_3)  =  z_3^3,
			\end{equation*}
	and	observe that 
			\begin{equation*}
					\Gamma_+(h)  \subsetneqq  \Gamma_+(f), \qquad
					\Gamma_+(l)  \subsetneqq  \Gamma_+(g),
			\end{equation*}
	by Corollary \ref{BrasseletIC2} we have 
				\begin{equation*}
					{\rm B}_{f_t,X_{\sigma}}(0) = {\rm B}_{f,X_{\sigma}}(0) 
					=  -3, \qquad
					{\rm B}_{f_t,X_{\sigma}^{g_s}}(0)  =  {\rm 
					B}_{f,X_{\sigma}^{g}}(0) =  12,
				\end{equation*}
	\noindent where $f_t(x)=f(x)+t \cdot h(x)$ is a deformation of the cusp 
	$f_0(z_1,z_2,z_3) = z_2^2 - z_1^3$ (see Figure 
	\eqref{fig:cuspidedeformation}) and $g_s(x)=g(x)+ s \cdot l(x)$. 
	Consequently,
	$$
	{\rm B}_{f_t,X_{\sigma}}(0) - {\rm 
	B}_{f_t,X_{\sigma}^{g_s}}(0) = -3 - 12 =-15,
	$$
 	for all $t, s \in \mathbb{C}$. 
%	\begin{figure}[h!]\centering
% \includegraphics[scale=0.5]{cuspidedeformation.eps}
% \caption{Cusp deformation $f_1$}\label{fig:cuspidedeformation}
%\end{figure}
\end{example}

%%%%%%%%%%%%%%%%%%%%%%%%%%%%%%%%%%%%%%%%%%%%%%%%%%%%%%%%%%%%%%%%%%%%%%%%%%%%%%%%%%%%%%%%%%5
%%%%%%%%%%%%%%%%%%%%%%%%%%%%%%%%%%%%%%%%%%%%%%%%%%%%%%%%%%%%%%%%%%%%%%%%%%%%%%%%%%%%%%%%%%%
\section{Indices of vector fields}\label{Section-IndicesVectorFields}
%%%%%%%%%%%%%%%%%%%%%%%%%%%%%%%%%%%%%%%%%%%%%%%%%%%%%%%5
%%%%%%%%%%%%%%%%%%%%%%%%%%%%%%%%%%%%%%%%%%%%%%%%%%%%%%%%

A toric surface $X_{\sigma}$, which is a cyclic quotient singularity,
always possesses a smoothing \cite[Satz 10]{Rie1}. Therefore, when we consider a radial continuous vector field $v$ on
$X_{\sigma}$ with an isolated singularity at $0$, we can relate the Euler 
characteristic of a fiber of this smoothing with the
$\rm{GSV}$ index of $v$ in $X_{\sigma}$. The definition of this index for 
smoothable isolated singularity can be found in \cite[Section 3]{BSS}.

In the particular, for the case of toric surfaces 
which are also isolated determinantal singularities, we have the following result concerning $\rm{GSV}$ index.

Let $X_{\sigma} \subset \mathbb{C}^n$ be a toric surface that is also an 
isolated determinantal singularity,\ie $\sigma$ is generated by the 
vectors $v_1 = p e_{1} - q e_{2}$ and $v_{2} = e_{2}$, where $0 < q < p$, $p,q$ are coprime,
and whose the Hirzebruch-Jung continued fraction is 
	$$
	\frac{p}{p-q} = [[a_2, 2, 2, \dots, 2, a_{n-1}]].
	$$
Consider $f_t(x) = f(x)+\sum_{j=1}^{r} \theta_{j}(t)\cdot h_j(x)$ a family of 
non-degenerate polynomial 
functions on $X_{\sigma}$, which satisfies the conditions
		\begin{equation*}
		\Gamma_{+}(h_j) \subsetneqq \Gamma_{+}(f), \ \ \text{for all} \ \ j=0,\dots,r.
		\end{equation*}
If this family has an isolated singularity at the origin, then the 
following result holds.

\begin{proposition}
	Let $v_t$ be the vector field given by the gradient of the function 
	$f_t$. Then, the following are equivalent:
	\item (a) ${\rm Eu}_{f_t,X_{\sigma}}(0)$ is constant for the family;
	
	\item (b) ${\rm Ind}_{GSV}(v_t,X_{\sigma},F)$ is constant for the family, 
	where $F$ is the flat map associated to the smoothing of $X_{\sigma}$.
\end{proposition}

\begin{proof}
	By \cite{BBT} the determinantal Milnor number of the 
	function $f$ on the Isolated Determinantal Singularity $X_{\sigma}$ is 
			$$
			\mu(f|_{X_{\sigma}}) = \# \Sigma(\tilde{f}|_{{X_{\sigma}}_s}),
			$$
	\noindent where ${X_{\sigma}}_s$ is a fiber of a smoothing of $X_{\sigma}$, $\tilde{f}|_{{X_{\sigma}}_s}$ is a morsefication of $f$ and $\# 
	\Sigma(\tilde{f}|_{{X_{\sigma}}_s})$ denote the number of Morse points of 
	$\tilde{f}$ on ${X_{\sigma}}_s$. From the definition of the GSV 
	index in the case of smoothable varieties (see \cite{BSS}) we have
	$$\mu(f|_{X_{\sigma}}) = {\rm Ind}_{GSV}(v,X_{\sigma},F).$$ Then the 
	proof follows by \cite{ABOT}, where it is proved that ${\rm 
	Eu}_{f_t,X_{\sigma}}(0)$ is constant for the 
	family if and only if $\mu(f_t|_{X_{\sigma}})$ is constant for the family.
\end{proof}

In \cite[Definition 2.5]{BMSS}, the authors extended the concept of $\rm{GSV}$ index and proved a L\^e-Greuel formula (see \cite[Theorem 
3.1]{BMSS}) which holds with the same hypotheses of Theorem \ref{Le Greuel}. 
However, in \cite{BMSS} the authors worked with the constructible function 
given by the characteristic function, while in \cite{NN} is considered the 
local Euler obstruction. Hence the $\rm{GSV}$ index and the Brasselet number 
are not related in general.

Assuming that $f_t$ is generically a submersion, for non-degenerate complete intersections, we have the following result.

\begin{proposition}\label{GSV} \LHchange{Troquei o corolario por proposicao.} \TDchange{ok}
	Let $S_{\sigma} = \mathbb{Z}_{+}^n$ and $X_{\sigma} = \mathbb{C}^n$ be the smooth $n$-dimensional toric variety. Let $(g,f): (X_{\sigma},0) \to 			(\mathbb{C}^2,0)$ be a non-degenerate complete intersection, such that $\mathcal{T}_g$ is a Whitney stratification of $X_{\sigma}^{g}$. If 
	$$\Bl g_s(x),f_t(x) \Br = \Bl g(x) + 
	\sum_{i=1}^{m}\xi_{i}(s)\cdot l_i(x), f(x)+\sum_{j=1}^{r} 
	\theta_{j}(t)\cdot h_j(x) \Br
	$$
	is a family of non-degenerate complete 
	intersections with $h_j$ and $l_i$ satisfying	the	condition 
	\eqref{condition}
	for all $i=1,\dots,m$ and $j=1,\dots,r$, such that $\mathcal{T}_{g_s}$ is a Whitney stratification of $X_{\sigma}^{g_{s}}$ and $g_s$ is prepolar 			with respect to $\mathcal{T}^{f_t}$ at the origin. Then, ${\rm Ind}_{GSV}(g_s, 0; f_t )$ is invariant to the family.
\end{proposition}

\begin{proof}
	In \cite[Section $5.2$]{BMSS}, the authors used \cite[Theorem 4.2]{NN} (considering the Euler characteristic as constructible function) to provide 		the following interpretation to the $\rm{GSV}$ index,
$$
\sum_{\Delta \prec \check{\sigma}} \big( \chi 
\big(T_{\Delta} \cap X_{\sigma} \cap B_{\varepsilon}(0) \cap f^{-1}(\delta) \big) - \chi \big(T_{\Delta} \cap X_{\sigma}^{g} \cap B_{\varepsilon}(0) \cap f^{-1}(\delta) \big) \big) = {\rm Ind}_{GSV}(g, 0; f ),
$$
	where ${\rm 
	Ind}_{GSV}(g, 0; f )$ is the GSV-index of $g$ on $X^{f}$ relative to the 
	function $f$ (see \cite[Definition 2.5]{BMSS}). Moreover, $f_t$ is a 
	family of non-degenerate polynomial functions, since $(g_s,f_t)$ 
	is 								non-degenerate complete intersections, 
	for all $s,t \in \mathbb{C}$. Then, to compute ${\rm Ind}_{GSV}(g, 0; f 
	)$ we apply \cite[Corollary 						3.5]{MT1} to the 
	first 
	term of the 			equality above and  to the 
	second term of the equality above we used \Eqref{nearbycicle1}, 
	\eqref{nearbycicle2}. Therefore, the 		result follows from the fact 
	that 
	$h_j$ 
	and $l_i$ 				satisfying	the	condition 				
	\eqref{condition}	for all $i=1,\dots,m$ and $j=1,\dots,r$.

%
%We just apply \cite[Corollary 3.5]{MT1} to the first term of the equality above and \cite[Theorem 3.12]{MT1} to the second term.

\end{proof}

\begin{example} 
	Consider the toric surface $X_{\sigma} = V(I_{\sigma}) \subset \mathbb{C}^3$, with 
	$I_{\sigma}$ the ideal generated by $z_1z_3 - z_2^{2}$. Let $f_t$ and $g_s$ be the same families of functions from Example \ref{examplemorse}, 			then
				$$
				{\rm Ind}_{GSV}(g_s, 0; f_t ) = -15
				$$
	\noindent for all $t, s \in \mathbb{C}$.
\end{example}

%\textcolor{red}{esse exemplo \'e referente a Proposi\c c\~ao 5.8}
%%%%%%%%%%%%%%%%%%%%%%%%%%%%%%%%
\section*{Acknowledgments}
%%%%%%%%%%%%%%%%%%%%%%%%%%%%%%%%
The authors are grateful to Nivaldo de G\'oes Grulha Jr. from ICMC-USP for helpful
conversations in developing this paper and to Bruna Or\'efice Okamoto from 
DM-UFSCar for helpful conversations about the Bruce-Roberts' Milnor number.
Through the project CAPES/PVE Grant  88881. 
068165/2014-01 of the program Science without borders, 
Professor Mauro Spreafico visited the DM-UFSCar in S\~ao Carlos providing 
useful discussions with the authors. Moreover, the
authors were partially supported by this project, therefore we are grateful 
to this program. We would like to thank the referee for 
many valuable suggestions which improved this paper.

\bibliography{DalHar}

\providecommand{\bysame}{\leavevmode\hbox to3em{\hrulefill}\thinspace}
\providecommand{\MR}{\relax\ifhmode\unskip\space\fi MR }
% \MRhref is called by the amsart/book/proc definition of \MR.
\providecommand{\MRhref}[2]{%
  \href{http://www.ams.org/mathscinet-getitem?mr=#1}{#2}
}
\providecommand{\href}[2]{#2}
\begin{thebibliography}{\textsc{ANnBOOT16}}

\bibitem[\textsc{ANnBOOT16}]{ABOT}
\textsc{D.~A.~H. Ament}, \textsc{J.~J. Nu\~no Ballesteros},
  \textsc{B.~Or\'efice-Okamoto}, and \textsc{J.~N. Tomazella}, \emph{The
  {E}uler obstruction of a function on a determinantal variety and on a curve},
  Bull. Braz. Math. Soc. (N.S.) \textbf{47} (2016), no.~3, 955--970.
  \MR{3549078}

\bibitem[\textsc{BLS00}]{BLS}
\textsc{J.-P. Brasselet}, \textsc{D.~T. L{\^e}}, and \textsc{J.~Seade},
  \emph{Euler obstruction and indices of vector fields}, Topology \textbf{39}
  (2000), no.~6, 1193--1208. \MR{1783853}

\bibitem[\textsc{BMPS04}]{BMPS}
\textsc{J.-P. Brasselet}, \textsc{D.~Massey}, \textsc{A.~J. Parameswaran}, and
  \textsc{J.~Seade}, \emph{Euler obstruction and defects of functions on
  singular varieties}, J. London Math. Soc. (2) \textbf{70} (2004), no.~1,
  59--76. \MR{2064752}

\bibitem[\textsc{BrGr10}]{BG}
\textsc{J.-P. Brasselet} and \textsc{N.~G. Grulha, Jr.}, \emph{Local {E}uler
  obstruction, old and new, {II}}, Real and complex singularities, London Math.
  Soc. Lecture Note Ser., vol. 380, Cambridge Univ. Press, Cambridge, 2010,
  pp.~23--45. \MR{2759085}

\bibitem[\textsc{BrRo88}]{BR}
\textsc{J.~W. Bruce} and \textsc{R.~M. Roberts}, \emph{Critical points of
  functions on analytic varieties}, Topology \textbf{27} (1988), no.~1, 57--90.
  \MR{935528}

\bibitem[\textsc{BrSc81}]{BS}
\textsc{J.-P. Brasselet} and \textsc{M.-H. Schwartz}, \emph{Sur les classes de
  {C}hern d'un ensemble analytique complexe}, The {E}uler-{P}oincar\'e
  characteristic ({F}rench), Ast\'erisque, vol.~82, Soc. Math. France, Paris,
  1981, pp.~93--147. \MR{629125}

\bibitem[\textsc{BSS09}]{BSS}
\textsc{J.-P. Brasselet}, \textsc{J.~Seade}, and \textsc{T.~Suwa}, \emph{Vector
  fields on singular varieties}, Lecture Notes in Mathematics, vol. 1987,
  Springer-Verlag, Berlin, 2009. \MR{2574165}

\bibitem[\textsc{BuGr80}]{BuG}
\textsc{R.-O. Buchweitz} and \textsc{G.-M. Greuel}, \emph{The {M}ilnor number
  and deformations of complex curve singularities}, Invent. Math. \textbf{58}
  (1980), no.~3, 241--281. \MR{571575}

\bibitem[\textsc{CBMSS16}]{BMSS}
\textsc{R.~Callejas-Bedregal}, \textsc{M.~F.~Z. Morgado}, \textsc{M.~Saia}, and
  \textsc{J.~Seade}, \emph{The {L}\^e-{G}reuel formula for functions on
  analytic spaces}, Tohoku Math. J. (2) \textbf{68} (2016), no.~3, 439--456.
  \MR{3550927}

\bibitem[\textsc{Dub81}]{Dubson}
\textsc{A.~Dubson}, \emph{Calcul des invariants num\'eriques des singularit\'es
  et applications}, Sonderforschungsbereich 40 Theoretische Mathemati,
  Universitaet {B}onn, 1981.

\bibitem[\textsc{DuGr14}]{NN}
\textsc{N.~Dutertre} and \textsc{N.~G. Grulha, Jr.}, \emph{L\^e-{G}reuel type
  formula for the {E}uler obstruction and applications}, Adv. Math.
  \textbf{251} (2014), 127--146. \MR{3130338}

\bibitem[\textsc{Dut16}]{N}
\textsc{N.~Dutertre}, \emph{Euler obstruction and {L}ipschitz-{K}illing
  curvatures}, Israel J. Math. \textbf{213} (2016), no.~1, 109--137.
  \MR{3509470}

\bibitem[\textsc{Ful93}]{F}
\textsc{W.~Fulton}, \emph{Introduction to toric varieties}, Annals of
  Mathematics Studies, vol. 131, Princeton University Press, Princeton, NJ,
  1993, The William H. Roever Lectures in Geometry. \MR{1234037}

\bibitem[\textsc{GGJR16}]{GGR}
\textsc{T.~Gaffney}, \textsc{N.~G. Grulha~Jr.}, and \textsc{M.~A.~S. Ruas},
  \emph{The local {E}uler obstruction and topology of the stabilization of
  associated determinantal varieties}, arXiv:1611.00749 (2016).

\bibitem[\textsc{GKZ08}]{GKZ}
\textsc{I.~M. Gelfand}, \textsc{M.~M. Kapranov}, and \textsc{A.~V. Zelevinsky},
  \emph{Discriminants, resultants and multidimensional determinants}, Modern
  Birkh\"auser Classics, Birkh\"auser Boston, Inc., Boston, MA, 2008, Reprint
  of the 1994 edition. \MR{2394437}

\bibitem[\textsc{GoMa88}]{GoMac}
\textsc{M.~Goresky} and \textsc{R.~MacPherson}, \emph{Stratified {M}orse
  theory}, Ergebnisse der Mathematik und ihrer Grenzgebiete (3) [Results in
  Mathematics and Related Areas (3)], vol.~14, Springer-Verlag, Berlin, 1988.
  \MR{932724}

\bibitem[\textsc{Gru09}]{Grulha2}
\textsc{N.~G. Grulha, Jr.}, \emph{The {E}uler obstruction and
  {B}ruce-{R}oberts' {M}ilnor number}, Q. J. Math. \textbf{60} (2009), no.~3,
  291--302. \MR{2533659}

\bibitem[\textsc{Gru12}]{Grulha1}
\bysame, \emph{Erratum: {T}he {E}uler obstruction and {B}ruce-{R}oberts'
  {M}ilnor number [mr2533659]}, Q. J. Math. \textbf{63} (2012), no.~1,
  257--258. \MR{2889190}

\bibitem[\textsc{GS79}]{Gonzalez}
\textsc{G.~Gonz\'alez-Sprinberg}, \emph{Calcul de l'invariant local d'{E}uler
  pour les singularit\'es quotient de surfaces}, C. R. Acad. Sci. Paris S\'er.
  A-B \textbf{288} (1979), no.~21, A989--A992. \MR{540374}

\bibitem[\textsc{Ham71}]{Hamm}
\textsc{H.~Hamm}, \emph{Lokale topologische {E}igenschaften komplexer
  {R}\"aume}, Math. Ann. \textbf{191} (1971), 235--252. \MR{0286143}

\bibitem[\textsc{L{\^e}73}]{Le2}
\textsc{D.~T. L{\^e}}, \emph{Calcul du nombre de cycles \'evanouissants d'une
  hypersurface complexe}, Ann. Inst. Fourier (Grenoble) \textbf{23} (1973),
  no.~4, 261--270. \MR{0330501}

\bibitem[\textsc{L{\^e}Te81}]{LT}
\textsc{D.~T. L{\^e}} and \textsc{B.~Teissier}, \emph{Vari\'et\'es polaires
  locales et classes de {C}hern des vari\'et\'es singuli\`eres}, Ann. of Math.
  (2) \textbf{114} (1981), no.~3, 457--491. \MR{634426}

\bibitem[\textsc{Mac74}]{Mac}
\textsc{R.~D. MacPherson}, \emph{Chern classes for singular algebraic
  varieties}, Ann. of Math. (2) \textbf{100} (1974), 423--432. \MR{0361141}

\bibitem[\textsc{Mas96}]{Massey}
\textsc{D.~B. Massey}, \emph{Hypercohomology of {M}ilnor fibres}, Topology
  \textbf{35} (1996), no.~4, 969--1003. \MR{1404920}

\bibitem[\textsc{Mas07}]{Massey1}
\bysame, \emph{Vanishing cycles and {T}hom's {$a_f$} condition}, Bull. Lond.
  Math. Soc. \textbf{39} (2007), no.~4, 591--602. \MR{2346940}

\bibitem[\textsc{MaTa11a}]{MT2}
\textsc{Y.~Matsui} and \textsc{K.~Takeuchi}, \emph{A geometric degree formula
  for {$A$}-discriminants and {E}uler obstructions of toric varieties}, Adv.
  Math. \textbf{226} (2011), no.~2, 2040--2064. \MR{2737807}

\bibitem[\textsc{MaTa11b}]{MT1}
\bysame, \emph{Milnor fibers over singular toric varieties and nearby cycle
  sheaves}, Tohoku Math. J. (2) \textbf{63} (2011), no.~1, 113--136.
  \MR{2788778}

\bibitem[\textsc{Mil68}]{M1}
\textsc{J.~Milnor}, \emph{Singular points of complex hypersurfaces}, Annals of
  Mathematics Studies, No. 61, Princeton University Press, Princeton, N.J.;
  University of Tokyo Press, Tokyo, 1968. \MR{0239612}

\bibitem[\textsc{NnBOOT13}]{BBT}
\textsc{J.~J. Nu\~no Ballesteros}, \textsc{B.~Or\'efice-Okamoto}, and
  \textsc{J.~N. Tomazella}, \emph{The vanishing {E}uler characteristic of an
  isolated determinantal singularity}, Israel J. Math. \textbf{197} (2013),
  no.~1, 475--495. \MR{3096625}

\bibitem[\textsc{NnBOOT18}]{BBT3}
\bysame, \emph{Non-negative deformations of weighted homogeneous
  singularities}, Glasg. Math. J. \textbf{60} (2018), no.~1, 175--185.
  \MR{3733838}

\bibitem[\textsc{NnBOT13}]{BBT2}
\textsc{J.~J. Nu\~no Ballesteros}, \textsc{B.~Or\'efice}, and \textsc{J.~N.
  Tomazella}, \emph{The {B}ruce-{R}oberts number of a function on a weighted
  homogeneous hypersurface}, Q. J. Math. \textbf{64} (2013), no.~1, 269--280.
  \MR{3032099}

\bibitem[\textsc{Oda88}]{O}
\textsc{T.~Oda}, \emph{Convex bodies and algebraic geometry---toric varieties
  and applications. {I}}, Algebraic {G}eometry {S}eminar ({S}ingapore, 1987),
  World Sci. Publishing, Singapore, 1988, pp.~89--94. \MR{966447}

\bibitem[\textsc{Oka90}]{Oka3}
\textsc{M.~Oka}, \emph{Canonical stratification of nondegenerate complete
  intersection varieties}, J. Math. Soc. Japan \textbf{42} (1990), no.~3,
  397--422. \MR{1056828}

\bibitem[\textsc{Oka97}]{Oka}
\bysame, \emph{Non-degenerate complete intersection singularity}, Actualit\'es
  Math\'ematiques. [Current Mathematical Topics], Hermann, Paris, 1997.
  \MR{1483897}

\bibitem[\textsc{OO11}]{BrunaTese}
\textsc{B.~Or\'efice-Okamoto}, \emph{O n\'umero de {Milnor} de uma
  singularidade isolada}, PhD Thesis, {U}niversidade {F}ederal de {S}\~ao
  {C}arlos (avaible in 23.09.2018 at
  https://www.dm.ufscar.br/ppgm/attachments/article/203/download.pdf), 2011.

\bibitem[\textsc{Par93}]{P}
\textsc{A.~Parusi\'nski}, \emph{Limits of tangent spaces to fibres and the
  {$w_f$} condition}, Duke Math. J. \textbf{72} (1993), no.~1, 99--108.
  \MR{1242881}

\bibitem[\textsc{Rie74}]{Rie1}
\textsc{O.~Riemenschneider}, \emph{Deformationen von
  {Q}uotientensingularit\"aten (nach zyklischen {G}ruppen)}, Math. Ann.
  \textbf{209} (1974), 211--248. \MR{0367276}

\bibitem[\textsc{Rie81}]{Rie2}
\bysame, \emph{Zweidimensionale {Q}uotientensingularit\"aten: {G}leichungen und
  {S}yzygien}, Arch. Math. (Basel) \textbf{37} (1981), no.~5, 406--417.
  \MR{643282}

\bibitem[\textsc{SRDSP14}]{MC}
\textsc{M.~A. Soares~Ruas} and \textsc{M.~Da~Silva~Pereira}, \emph{Codimension
  two determinantal varieties with isolated singularities}, Math. Scand.
  \textbf{115} (2014), no.~2, 161--172. \MR{3291723}

\bibitem[\textsc{STV05}]{STV}
\textsc{J.~Seade}, \textsc{M.~Tib{\u a}r}, and \textsc{A.~Verjovsky},
  \emph{Milnor numbers and {E}uler obstruction}, Bull. Braz. Math. Soc. (N.S.)
  \textbf{36} (2005), no.~2, 275--283. \MR{2152019}

\bibitem[\textsc{Var76}]{Varchenko}
\textsc{A.~N. Varchenko}, \emph{Zeta-function of monodromy and {N}ewton's
  diagram}, Invent. Math. \textbf{37} (1976), no.~3, 253--262. \MR{0424806}

\end{thebibliography}
\bibliographystyle{amsalpha-lmp}

\end{document}